\documentclass[review,hidelinks,onefignum,onetabnum]{siamart220329}

\usepackage{lipsum}
\usepackage{amsfonts}
\usepackage{graphicx}
\usepackage{graphbox}
\usepackage{epstopdf}
\usepackage{algorithmic}
\ifpdf
  \DeclareGraphicsExtensions{.eps,.pdf,.png,.jpg}
\else
  \DeclareGraphicsExtensions{.eps}
\fi

\newsiamremark{remark}{Remark}
\newsiamremark{hypothesis}{Hypothesis}
\crefname{hypothesis}{Hypothesis}{Hypotheses}
\newsiamthm{claim}{Claim}

\headers{Robust Stokes-Darcy preconditioning}{X.Hu, M.Kuchta, K.-A.Mardal, X.Wang}

\title{Parameter-robust preconditioner for Stokes-Darcy coupled problem with Lagrange multiplier\thanks{Submitted to the editors DATE.
\funding{MK acknowledges support from the Research Council of Norway (NFR) grant \#303362.
KAM acknowledges funding and support from Stiftelsen
Kristian Gerhard Jebsen via the K. G. Jebsen Centre for Brain Fluid
Research, the Research Council of Norway  \#300305 (SciML) and \#301013 (Alzheimer´s physics) and the European Research Council under grant  101141807 (aCleanBrain).
}}}

\author{Xiaozhe Hu\thanks{Department of Mathematics, Tufts University, Medford, MA 02155, USA
  (\email{xiaozhe.hu@tufts.edu}).}
\and Miroslav Kuchta\thanks{Department of Numerical Analysis and Scientific Computing, Simula Research Laboratory, 0164 Oslo, Norway
  (\email{miroslav@simula.no}).}
\and Kent-Andre Mardal\thanks{Simula Research Laboratory, 0164 Oslo, Norway, and Department of Mathematics, University of Oslo, Blindern, 0316 Oslo, Norway
  (\email{kent-and@uio.no})}
\and Xue Wang \thanks{School of Mathematics, Shandong University, 250100 Jinan, China}
  (\email{wangxsdu@mail.sdu.edu.cn})
  }

\usepackage{amsopn}

\usepackage[]{hyperref}
\hypersetup{
    colorlinks = true,
    citecolor=blue,
    linkcolor=cyan!80!purple
}

\ifpdf
\hypersetup{
  pdftitle={Parameter-robust preconditioner for Stokes-Darcy coupled problem with Lagrange multiplier},
  pdfauthor={Xiaozhe Hu, Miroslav Kuchta, Kent-Andre Mardal, and Xue Wang}
}
\fi

\usepackage{bm}
\usepackage{amssymb}
\usepackage{tikz}
\usepackage{caption}
\usepackage{subcaption}
\usepackage{multirow}
\usetikzlibrary{calc}
\usepackage{geometry}
\nolinenumbers

\definecolor{darkgreen}{rgb}{0, .6, 0}

\newtheorem{example}{Example}[section]

\begin{document}
\maketitle
\numberwithin{equation}{section}

\begin{abstract}
    In this paper, we propose a parameter-robust preconditioner for the coupled Stokes-Darcy problem equipped with various boundary conditions, enforcing the mass conservation at the interface via a Lagrange multiplier. We rigorously establish that the coupled system is well-posed with respect to physical parameters and mesh size and provide a framework for constructing parameter-robust preconditioners. Furthermore, we analyze the convergence behavior of the Minimal Residual method in the presence of small outlier eigenvalues linked to specific boundary conditions,  which can lead to slow convergence or stagnation.  To address this issue, we employ deflation techniques to accelerate the convergence. Finally, numerical experiments confirm the effectiveness and robustness of the proposed approach.
\end{abstract}

\begin{keywords}
  Stokes-Darcy problem; Lagrange multiplier; preconditioning; parameter-robust solver; deflation.
\end{keywords}

\begin{AMS}
  65F08, 65F10, 65M60, 65N12, 65N55
\end{AMS}

\section{Introduction}
This paper focuse on the preconditioning of a monolithic scheme for the coupled Stokes-Darcy flow problem, which models the interaction between free fluid and porous media flows. It has been extensively studied in recent years due to its numerous applications, such as industrial filtration, water/gas management in fuel cells, and surface–subsurface interactions \cite{cao2010coupled,ervin2009coupled}.  The classical set of coupling conditions contain continuity of normal velocity and balance of normal and tangential forces, with the latter being described by the Beavers-Joseph-Saffman condition, see \cite{Beavers1967BoundaryCA,Saffman1971On}.

In this paper, we focus on the coupled, linear, stationary Stokes-Darcy model and the formulation where the coupling is enforced using Lagrange multiplier on the interface \cite{layton2002coupling}. Different formulations of the Stokes-Darcy problem have been analyzed in the literature, e.g. the formulation with Darcy problem in the primal form \cite{discacciati2004domain} or fully mixed formulation \cite{gatica2011analysis}. The formulations have been targeted by discretization schemes, for example,  finite volume method \cite{fetzer2017conditions,masson2016coupling,schneider2020coupling}, (mixed) finite element method \cite{CAUCAO2017943,CaucaoGaticaOyarzúaŠebestová,rui2009unified}, and finite difference method \cite{shiue2018convergence,greif2023block,rui2020mac}. Iterative methods for the resulting linear systems have been developed including domain decomposition approaches \cite{discacciati2007robin,vassilev2014domain, boon2020parameter}, multilevel methods \cite{mu2007two, luo2017uzawa} or block preconditioners \cite{boon2022robust,Holter2020robust, Strohbeck2023Robust,Strohbeck2024Efficient}. 
In particular, Boon et al. \cite{boon2022robust} developed robust monolithic solvers for the coupled primal Stokes–Darcy problem, assuming that the interface is in contact with portions of the external boundary where Neumann conditions are imposed on both subdomains. However, the authors also noted important differences in the norm of the Lagrange multiplier at the interface compared to the mixed-form Stokes–Darcy problem studied in \cite{Holter2020robust}, and extensive numerical tests were presented to support this observation.

To our knowledge, an optimal monolithic preconditioner, designed using the operator preconditioning framework and robust with respect to both physical and discretization parameters, has not been accomplished for the Lagrange multiplier formulation of the coupled Stokes-Darcy problem. In this work,  we rigorously establish the well-posedness of the problem, paying in particular special attention to the boundary conditions of the subproblems (cf. numerical experiments in \cite{Holter2020robust}).
As will become evident the coupled problem is not well-posed uniformly with respect to the parameters in certain boundary condition configurations. These cases are characterized by existence of a near-kernel space. In turn, in our analysis, we examine the inf-sup condition, with particular focus on how the inf-sup constant depends on the physical parameters, the entire solution space, and the complement of the near-kernel space.

Using stability estimates from the continuous problem, we will analyze the convergence of the Minimal Residual (MINRES) method. It is well-known that the convergence of such a Krylov subspace iterative method is closely related to the spectral properties of the system matrix~\cite{barrett1994templates,saad2003iterative, nielsen2013analysis,simoncini2013superlinear}. In particular, achieving fast convergence requires clustered eigenvalues that are bounded away from zero. However, in near-kernel cases of the Stokes-Darcy problem, the MINRES method exhibits slow convergence or even stagnation, meaning the residual remains nearly unchanged over a substantial number of iterations, depending on the physical parameter values.  To better understand this behavior, we analyze the convergence of the MINRES method in the presence of a ``bad'' or slow-to-converge eigenvalue of the preconditioned matrix. Building on this analysis, we employ the deflation technique \cite{gutknecht2012spectral,soodhalter2020surveysubspacerecyclingiterative}, which is commonly used to shift isolated clusters of small eigenvalues, resulting in a tighter eigenvalue distribution and a reduced condition number.  In connection with the Stokes flow,  \cite{Dumitrasc_2024} applied deflation preconditioners to speed up the slow convergence of standard block preconditioners when the problem domain is elongated, see also \cite{dobrowolski2003lbb}.

The remainder of the paper is organized as follows. In \Cref{sec2:model}, we introduce some notations and present the Stokes-Darcy coupled model, its weak formulation, and the abstract framework for establishing  well-posedness. The well-posedness analysis of the coupled Stokes-Darcy problem is carried out in \Cref{sec3:well-posed}, where we furthermore develop the parameter-robust preconditioner with respect to both physical and discretization parameters. In \Cref{sec4:discrete analysis}, we analyze the behavior of the MINRES method and make use of the deflation technique to accelerate convergence . Finally, we provide conclusions in \Cref{sec:conclusions}.

\section{Preliminaries and  Model}\label{sec2:model}
In this section, we first introduce the required notation, and then present the Stokes-Darcy model and derive weak formulation of the problem. Additionally, we outline the abstract framework proposed in \cite{Holter2020robust} for analyzing the well-posedness of general coupled problems, emphasizing the key assumptions and functional settings required to ensure stability and solvability of the system.

\subsection{Preliminaries}
Throughout the paper, we use boldface symbols to represent vector fields and their corresponding spaces, while scalar fields and their spaces are written in regular font.

For a domain $\Omega_i\subset \mathbb{R}^d,  d\in\{2, 3\}$, the Sobolev space $H^m(\Omega_i) = W^{2,m}(\Omega_i)$ is defined in the usual way, with the norm $\|\cdot\|_{m,i}$ and seminorm $|\cdot|_{m,i}$, for any $i = \{S,D\}$. When $m=0$, we identify $H^{0}(\Omega_i) = L^2(\Omega_i)$ and use $\|\cdot\|_i$ for $\|\cdot\|_{0,i}$, with associated inner product denoted by $(\cdot, \cdot)_i$.  For a general Hilbert space $V$, we denote its norm by $\lVert \cdot \rVert_V$. For the $L^2$ space, we use the simplified notation $\lVert\cdot\rVert$ for the norm and $(\cdot, \cdot)$ for the corresponding inner product. To emphasize the domain $\Omega$, we may explicitly write $(\cdot, \cdot)_{\Omega}$ and $\|\cdot\|_{\Omega}$. Finally, function spaces consisting of functions with mean value zero are represented as quotient spaces, e.g., $L^2(\Omega)/\mathbb{R}$ denotes the space of $L^2$ functions on $\Omega$ with zero mean. 

If $X, Y$ are Sobolev spaces and $\alpha$ is an arbitrary positive real number, we define the weighted space $\alpha X$ as the space $X$ equipped with the norm $\alpha\|\cdot\|_X$. The intersection $X \cap Y$ and sum $X+Y$ form Banach spaces with norms given by
$$
\|u\|_{X \cap Y}=\sqrt{\|u\|_X^2+\|u\|_Y^2}, \quad \text { and } \quad\|u\|_{X+Y}=\inf _{\substack{x+y=u \\ x \in X, y \in Y}} \sqrt{\|x\|_X^2+\|y\|_Y^2}.
$$
The dual space of the Hilbert space $X$ is denoted by $X^{'}$  and $\mathcal{L}(X,Y)$ denotes the space of bounded linear operators mapping from $X$ to $Y$.  

For $\Gamma$ a subset of $\partial \Omega$, we define $H_{00}^{1 / 2}(\Gamma)$ as the space of all $w \in$ $H^{1 / 2}(\Gamma)$ whose extension by zero to $\partial \Omega$ belongs to $H^{1 / 2}(\partial \Omega)$. We also define $H_{00}^{-1 / 2}(\Gamma)$ to be the dual of $H_{00}^{1 / 2}(\Gamma)$. For the details on fractional and trace spaces, we refer to \cite{juan2007nonmathcing}. 

Throughout this paper, the constant $C$,  with or without subscripts,  represents a constant independent of the material and discrete parameters.  Additionally, we use the notation  $a\lesssim b$ $(a\gtrsim b)$ to indicate $a\leq Cb$ $(a\geq Cb)$ for some constant $C$.

\subsection{Model}
Let $\Omega = \Omega_S\cup\Omega_D\subset\mathbb{R}^d, d\in\{2,3\}$ be an open, bounded domain, with $\Omega_S$, $\Omega_D$ as two non-overlapping Lipschitz subdomains sharing a common interface $\Gamma = \partial\Omega_S\cap\partial\Omega_D\subset\mathbb{R}^{d-1}$. We define the outer boundaries $\Gamma_S = \partial\Omega_S\setminus\Gamma$ and $\Gamma_D = \partial\Omega_D\setminus\Gamma$, and assume that $\lvert \Gamma_S \rvert>0$ and $\lvert \Gamma_D \rvert>0$. The subscripts $S$ and $D$ are used throughout this work to denote entities related to Stokes and Darcy domains, respectively. Let $\mathbf{n}_i$ denote the outward unit vector normal to $\partial\Omega_i$, $i = S, D$. Consider the Stokes problem in the free fluid domain $\Omega_S$,
\begin{equation}\label{Stokes-model}
\begin{cases}
-\nabla\cdot\bm{\sigma}(\mathbf{u}_S,p_S)=\bm{f}_S, \quad \nabla \cdot \mathbf{u}_S = 0, & \text { in } \Omega_S, \\
\mathbf{u}_S=\mathbf{0}, & \text { on } \Gamma_S^E, \\
-\bm{\sigma}(\mathbf{u}_S,p_S)\cdot\mathbf{n}_S = \mathbf{0}, & \text { on } \Gamma_S^N.
\end{cases}
\end{equation}
Here,  the unknowns $\mathbf{u}_S, p_S$ are the velocity and pressure, respectively. $\bm{\sigma}(\mathbf{u}_S,p_S) = 2 \mu \bm{\epsilon}\left(\mathbf{u}_S\right)- p_S I $ denotes the stress tensor where $\bm\epsilon(\mathbf{u}_S)=\frac{1}{2}\left(\nabla \mathbf{u}_S+\nabla \mathbf{u}_S^T\right)$ is the deformation rate tensor and $\mu>0$ denotes the fluid viscosity. The source term of the momentum equation is denoted as $\bm{f}_S$.

In the porous medium $\Omega_D$, the fluid pressure $p_D$ and velocity $\mathbf{u}_D$ satisfy the Darcy's law
\begin{equation}\label{Darcy-model}
\begin{cases}
\mathbf{K}^{-1} \mathbf{u}_D+\nabla p_D = \mathbf{0}, \quad
\nabla \cdot \mathbf{u}_D = g_D, & \text { in } \Omega_D, \\
\mathbf{u}_D \cdot \mathbf{n}_D = 0, & \text { on } \Gamma_D^E, \\
p_D = 0, & \text { on } \Gamma_D^N.
\end{cases}
\end{equation}
In this setting, $\mathbf{K}$ is the hydraulic conductivity of the porous medium, which is a positive-definite tensor, and $g_D$ is the given source term. For simplicity, we let $\mathbf{K} = K\mathbf{I}$, where $K$ is a scalar constant, and assume throughout the paper that all the material parameters are constant.

The problems \eqref{Stokes-model}-\eqref{Darcy-model} are coupled by the following interface conditions
\begin{subequations}\label{interface}
\begin{alignat}{2}
\mathbf{u}_S \cdot \mathbf{n}_S+\mathbf{u}_D \cdot \mathbf{n}_D & = 0, \quad & \text{ on } \Gamma, \label{interface-1}\\
-2 \mu \bm{\epsilon}\left(\mathbf{u}_S\right) \mathbf{n}_S \cdot \mathbf{n}_S + p_S & = p_D,  \quad & \text{ on } \Gamma, \label{interface-2} \\
-2 \mu \bm{\epsilon}\left(\mathbf{u}_S\right) \mathbf{n}_S \cdot \bm{\tau} & = \beta_\tau\mathbf{u}_S \cdot \bm{\tau}, \quad & \text{ on } \Gamma. \label{interface-3}
\end{alignat}
\end{subequations}
On the interface $\Gamma$, we have $\mathbf{n}_S=-\mathbf{n}_D$, and $\bm{\tau} = \mathbf{I} - (\mathbf{n}_S\otimes\mathbf{n}_S)$ denotes a tangential projection onto $\Gamma$. We also define $\beta_\tau = \alpha_{\text{BJS}}\frac{\mu}{\sqrt{\bm{\tau}\cdot\bm{\kappa}\bm{\tau}}}$, where $\bm{\kappa} = \mu\mathbf{K}$ is the permeability, and $\alpha_{\text{BJS}}>0$ is the Beavers-Joseph-Saffman (BJS) slip coefficient. As for the physical meaning of the interface conditions, the first interface condition represents the mass conservation across $\Gamma$, the second one represents the balance of normal forces, and the third condition is the BJS condition \cite{Beavers1967BoundaryCA,Saffman1971On}. 

On the outer boundaries, we consider the decomposition $\Gamma_S = \Gamma_S^E \cup \Gamma_S^N$, and $\Gamma_D = \Gamma_D^E \cup \Gamma_D^N$, where the superscript distinguishes essential $(E)$ and natural $(N)$ boundary conditions. Specifically, on $\Gamma_S^E$ and $\Gamma_D^E$, we prescribe the velocity $\mathbf{u}_S$ and normal velocity $\mathbf{u}_D\cdot\mathbf{n}_D$, respectively. On $\Gamma_S^N$ and $\Gamma_D^N$, we prescribe the traction $\bm{\sigma}\cdot\mathbf{n}_S$ and the pressure, $p_D$, respectively. To avoid the rigid body motion, we assume that $|\Gamma_S^{E}|>0$.

\subsection{Weak formulation}
To present the weak form of the coupled problem, let us define the appropriate weighted Sobolev spaces. Let the velocity and pressure spaces on $\Omega_S$ equipped with $\Gamma_S^E,\, \Gamma_S^N \neq \emptyset$ be
\begin{align}\label{Stokes-spaces}
   & \mathbf{V}_S := \sqrt{\mu}\,\mathbf{H}_{0,E}^1(\Omega_S) \cap \sqrt{\beta_{\tau}}\mathbf{L}_t^2(\Gamma), \quad Q_S := \frac{1}{\sqrt{\mu}}L^2(\Omega_S),
\end{align}
where 
  $\mathbf{H}_{0,E}^1(\Omega_S) := \left\{\mathbf{v}\in[H^1(\Omega_S)]^d \, | \ \mathbf{v} = 0 \text{ on } \Gamma_S^E \right\}$ and $\mathbf{L}^2_t(\Gamma) = \{\, \mathbf{u} \in [L^2(\Gamma)]^d \, | \ \mathbf{u} \cdot \bm{\tau} \in L^2(\Gamma) \,\}$.
The velocity-pressure spaces on $\Omega_D$ are 
\begin{align}\label{Darcy-spaces}
   & \mathbf{V}_D := \frac{1}{\sqrt{K}} \mathbf{H}_{0,E}(\operatorname{div};\Omega_D), \quad Q_D := \sqrt{K}L^2(\Omega_D),
\end{align}
where $\mathbf{H}_{0,E}(\operatorname{div};\Omega_D) := \left\{\mathbf{v}\in [L^2(\Omega_D)]^d | \ \nabla\cdot\mathbf{v}\in L^2(\Omega_D), \  \mathbf{v}\cdot\mathbf{n}_D = 0 \text{ on } \Gamma_D^E \right\}$. 
In addition, let $\mathbf{V} = \mathbf{V}_S\times\mathbf{V}_D$ and $Q = Q_S\times Q_D$. For any $\mathbf{v}\in\mathbf{V},\, q\in Q$, we denote $\mathbf{v}_i = \mathbf{v}|_{\Omega_i},\, q_i = q|_{\Omega_i},\, i = S,D$. In particular, if we consider the coupled problem with pure Dirichlet/essential boundary conditions, that is, $\Gamma_S^N = \Gamma_D^N = \emptyset$, then the pressure space should be redefined as
$
    Q/\mathbb{R} = \left\{q=(q_S,q_D)\in Q_S\times Q_D| \int_{\Omega} q \mathrm{~d}x = 0
    \right\}.
$
Following \cite{Holter2020robust}, the weighted norms of $\mathbf{V}$ and $Q$ are given by
\begin{align*}
     \|\mathbf{v}\|_{\mathbf{V}} &  := \left(\|\mathbf{v}_S\|_{\mathbf{V}_S}^2  +\|\mathbf{v}_D\|_{\mathbf{V}_D}^2\right)^{1/2}, \quad 
    \|q \|_{Q} := \left(\|q_S\|_{Q_S}^2+ \|q_D \|_{Q_D}^2\right)^{1/2},
\end{align*}
where
\begin{align*}
\|\mathbf{v}_S\|_{\mathbf{V}_S}&:= \left( \mu\|\mathbf{v}_S\|_{1,S}^2 + \beta_{\tau}\|\mathbf{v}_S\cdot\bm{\tau}\|_{\Gamma}^2 \right)^{1/2}, \quad  \|\mathbf{v}_D\|_{\mathbf{V}_D}:= \left( {K}^{-1}\|\mathbf{v}_D\|_{D}^2 + {K}^{-1}\|\nabla\cdot\mathbf{v}_D\|_{D}^2\right)^{1/2}, \\
\|q_S\|_{Q_S} &:= \mu^{-1/2}\|q_S\|_{S}, \quad \|q_D \|_{Q_D} :=  {K}^{1/2}\|q_D\|_{D}.
\end{align*}
Next, we introduce the Lagrange multiplier $\lambda$ to enforce the mass conservation condition \eqref{interface-1} and the balance of stress \eqref{interface-2} on the interface. Specifically, we define $\lambda:=p_D$ on $\Gamma$. Here, $\lambda\in \Lambda$, where the multiplier space $\Lambda$ will be defined later, depending on different boundary conditions configurations.

With the above function spaces, we arrive at the Lagrange multiplier variational formulation: find the $(\mathbf{u}, p, \lambda)\in \mathbf{V}\times Q\times \Lambda$ such that
\begin{subequations}\label{weak-form}
    \begin{alignat}{2}
        a(\mathbf{u},\mathbf{v}) - b(\mathbf{v},(p,\lambda)) & = (\bm{f}_S,\mathbf{v}_S),  \quad && \forall\,\mathbf{v}\in\mathbf{V}, \\
        b(\mathbf{u},(q,\phi)) & = (g_D, q_D), \quad && \forall\, (q,\phi)\in Q\times\Lambda,
    \end{alignat}
\end{subequations}
with the bilinear forms defined as follows,
\begin{align*}
    & a(\mathbf{u},\mathbf{v}) = 2\mu(\bm{\epsilon}(\mathbf{u}_S), \bm{\epsilon}(\mathbf{v}_S))_{S} + \beta_\tau(T_t\mathbf{u}_S, T_t\mathbf{v}_S)_{\Gamma} + {K}^{-1}(\mathbf{u}_D,\mathbf{v}_D)_{D}, \\
    & b(\mathbf{v}, (p,\lambda)) =  ( \nabla \cdot \mathbf{v}_S, p_S)_{S} +(\nabla\cdot\mathbf{v}_D, p_D)_{D} - ( T_n\mathbf{v}_S, \lambda)_{\Gamma} + (T_n\mathbf{v}_D, \lambda )_{\Gamma}.
\end{align*}
Here, $T_t$ is the tangential trace operator and $T_n$ is the normal trace operator such that,
\[
T_t\mathbf{u} = \mathbf{u}\cdot \bm{\tau}, \quad T_n\mathbf{u} = \mathbf{u}\cdot \mathbf{n}_S, \quad \mbox{ on } \Gamma.
\]
To distinguish $T_n\mathbf{u}_S$ and $T_n\mathbf{u}_D$, we introduce $T_S\mathbf{u}_S = T_n\mathbf{u}_S$ and $T_D\mathbf{u}_D =  -T_n\mathbf{u}_D$, which will be used in the following abstract framework.
The weak form \cref{weak-form} induces the following linear operator
\begin{equation}\label{operator-A}
\mathcal{A}=\left(\begin{array}{cc|ccc}
-2\mu\nabla\cdot\bm{\epsilon} + \beta_{\tau} T_t^{\prime} T_t & & -\nabla & & T_S^{\prime} \\ 
& {K}^{-1} I & & -\nabla & T_D^{\prime} \\ 
\hline \nabla \cdot & & & & \\ 
& \nabla \cdot & & & \\ 
T_S & T_D & & &
\end{array}\right).
\end{equation}

\subsection{Abstract framework}\label{sec:abstract}
In this section, we review general framework for the interface coupled multiphysics problems. Given two non-overlapping domains with common interface $\Gamma$, we assume that there are  two subdomain problems, each with boundary conditions on $\Gamma$ enforced in terms of a Lagrange multiplier. Consequently, we are concerned with a pair or saddle-point problems ($i = S, D$) of the form:
find $(\mathbf{u}_i,p_i,\lambda_i)\in \mathbf{V}_i\times Q_i\times \Lambda_i$ such that 
\begin{equation}\label{abstrac-equ}
\left(\begin{array}{ccc}
A_i & B_i^{\prime} & T_i^{\prime} \\
B_i &            &       \\
T_i &            &        \\
\end{array}\right)
\left(\begin{array}{c}
     \mathbf{u}_i \\
      p_i\\
      \lambda_i \\
\end{array}\right)
=\left(\begin{array}{c}
     \mathbf{f}_i \\
      g_i\\
      h_i \\
\end{array}\right)
\in\left(\begin{array}{c}
     \mathbf{V}_i^{\prime} \\
      Q_i^{\prime}\\
      \Lambda_i^{\prime}\\
\end{array}\right).
\end{equation}
According to the Brezzi theory \cite{brezzi1974existence}, the well-posedness of the \cref{abstrac-equ} requires the existence of constants $\alpha_i, C_i, \gamma_i, D_i>0$ such that 
\begin{subequations}\label{Brezzi}
\begin{align}
    & (A_i\mathbf{u}_i,\mathbf{u}_i) \geq \alpha_i\|\mathbf{u}_i\|_{\mathbf{V}_i}^2, \quad \forall\,\mathbf{u}_i\in \mathbf{Z}_i, \label{a-coercivity} \\
    & (A_i\mathbf{u}_i,\mathbf{v}_i) \leq C_i\|\mathbf{u}_i\|_{\mathbf{V}_i} \|\mathbf{v}_i\|_{\mathbf{V}_i}, \quad \forall\,\mathbf{u}_i,\mathbf{v}_i\in \mathbf{V}_i,  \label{a-continuity} \\
    & \sup_{\mathbf{0}\neq\mathbf{u}_i\in \mathbf{V}_i}\frac{(B_i\mathbf{u}_i, q_i) + (T_i\mathbf{u}_i, \phi_i)}{\|\mathbf{u}_i\|_{\mathbf{V}_i}} \geq \gamma_i \left(\|q_i\|_{Q_i}^2 + \|\phi_i\|_{\Lambda_i}^2\right)^{1/2}, \quad \forall\,(q_i,\phi_i) \in Q_i\times\Lambda_i,\label{inf-sup}\\
    & (B_i\mathbf{u}_i,q_i) + (T_i\mathbf{u}_i,\phi_i) \leq D_i \|\mathbf{u}_i\|_{\mathbf{V}_i} \left( \|q_i\|_{Q_i}^2 + \|\phi_i\|_{\Lambda_i}^2  \right)^{1/2}, \quad \forall\,(\mathbf{u}_i, q_i,\phi_i) \in \mathbf{V}_i\times Q_i\times \Lambda_i,\label{b-continuity}
\end{align}
\end{subequations}
where $\mathbf{Z}_i = \left\{ \mathbf{u}_i\in\mathbf{V}_i\ | \ (B_i\mathbf{u}_i, q_i) + (T_i\mathbf{u}_i, \lambda_i) = 0, \quad \forall\,(q_i, \lambda_i)\in Q_i\times \Lambda_i  \right\}$ is the kernel space. We shall consider the following condition, 
\begin{equation}\label{a-coercivity-stronger}
    (A_i\mathbf{u}_i,\mathbf{u}_i) \geq \alpha_i \|\mathbf{u}_i\|_{\mathbf{V}_i}^2, \quad \forall\,\mathbf{u}_i\in\left\{ \mathbf{u}_i\in\mathbf{V}\ | \ (B_i\mathbf{u}_i, q_i) = 0, \ \forall\,q_i\in Q_i  \right\},
\end{equation}
which is stronger than \eqref{a-coercivity}, but more common in single-physics problems with the boundary conditions $T_i\mathbf{u}_i=0$ on $\Gamma$ enforced in the standard way. 
The Brezzi conditions ensure that 
\begin{align}\label{eq:estimate-abstract}
\left\|\mathbf{u}_i\right\|_{\mathbf{V}_i}+\left\|p_i\right\|_{Q_i}+\left\|\lambda_i\right\|_{\Lambda_i} \leq E_i(\left\|\mathbf{f}_i\right\|_{\mathbf{V}_i^{\prime}}+\left\|g_i\right\|_{Q_i^{\prime}}+\left\|h_i\right\|_{\Lambda_i^{\prime}}).
\end{align}
Here, $E_i$ depends only on $\alpha_i, \gamma_i, C_i$ and $D_i$.

We now consider the existence and uniqueness of the coupled problem: find $\left(\mathbf{u}_S, \mathbf{u}_D, p_S, p_D, \lambda\right) \in \mathbf{V}_S \times \mathbf{V}_D \times Q_S \times Q_D \times (\Lambda_S \cap \Lambda_D)$ such that
\begin{equation}\label{abstract-coupled-equ}
\mathcal{A}\left(\begin{array}{c}
\mathbf{u}_S \\
\mathbf{u}_D \\
p_S \\
p_D \\
\lambda
\end{array}\right)=\left(\begin{array}{lllll}
A_S & & B_S^{\prime} & & T_S^{\prime} \\
& A_D & & B_D^{\prime} & T_D^{\prime} \\
B_S & & & & \\
& B_D & & & \\
T_S & T_D & & &
\end{array}\right)\left(\begin{array}{c}
\mathbf{u}_S \\
\mathbf{u}_D \\
p_S \\
p_D \\
\lambda
\end{array}\right)=\left(\begin{array}{c}
\bm{f}_S \\
\bm{f}_D \\
g_S \\
g_D \\
h
\end{array}\right) \in\left(\begin{array}{c}
\mathbf{V}_S^{\prime} \\
\mathbf{V}_D^{\prime} \\
Q_S^{\prime} \\
Q_D^{\prime} \\
\Lambda_S^{\prime}+\Lambda_D^{\prime}
\end{array}\right) .
\end{equation}
We remark that $T_i: \mathbf{V}_i \rightarrow \Lambda_i^{\prime}$ and hence $T_S \mathbf{u}_S+T_D \mathbf{u}_D \in \Lambda_S^{\prime}+\Lambda_D^{\prime}$. Therefore, $\lambda \in \Lambda_S \cap \Lambda_D$ since $\left(\Lambda_S \cap \Lambda_D\right)^{\prime}=\Lambda_S^{\prime}+\Lambda_D^{\prime}$.   Following Theorem 3.1 in \cite{Holter2020robust}, we have the following theorem.
\begin{theorem}[Well-posedness of \eqref{abstract-coupled-equ} ]\label{thm-abstract}
    Suppose the problems \eqref{abstrac-equ} satisfy the Brezzi conditions \eqref{a-continuity}-\eqref{b-continuity} in $\mathbf{V}_i \times Q_i \times \Lambda_i, i=S,D$ and the coercivity condition \eqref{a-coercivity-stronger}. Then the coupled problem \eqref{abstract-coupled-equ} is well posed in $\mathbf{W} = \mathbf{V}_S \times Q_S \times \mathbf{V}_D \times Q_D \times\left(\Lambda_S \cap \Lambda_D\right)$ in the sense that
$$
\|\mathcal{A}\|_{\mathcal{L}\left(\mathbf{W}, \mathbf{W}^{\prime}\right)} \quad \text { and } \quad\left\|\mathcal{A}^{-1}\right\|_{\mathcal{L}\left(\mathbf{W}^{\prime}, \mathbf{W}\right)}
$$
are bounded by some positive constant $C$ depending only on the Brezzi constants of problems \eqref{abstrac-equ}.
\end{theorem}
\begin{proof}
We verify the Brezzi conditions for \cref{abstract-coupled-equ} in the form
\begin{equation*}
\mathcal{A}=\left(\begin{array}{cc}
A & B^{\prime} \\
B &
\end{array}\right), \text { where } A=\left(\begin{array}{cc}
A_S & \\
& A_D
\end{array}\right) \quad \text { and } \quad B=\left(\begin{array}{cc}
B_S & \\
& B_D \\
T_S & T_D
\end{array}\right).
\end{equation*}
By considering $\mathcal{A}$ as an operator on $\mathbf{V} \times Q$ where $\mathbf{V}=\mathbf{V}_S \times \mathbf{V}_D$ and $Q=$ $Q_S \times Q_D \times\left(\Lambda_S \cap \Lambda_D\right)$. The boundedness of $A$ follows from \cref{a-continuity}  and the Cauchy-Schwarz inequality: for any $\mathbf{u}_S, \mathbf{v}_S \in \mathbf{V}_S, \mathbf{u}_D, \mathbf{v}_D \in \mathbf{V}_D$ we have that
\begin{align*}
\left(A\left(\mathbf{u}_S, \mathbf{u}_D\right),\left(\mathbf{v}_S, \mathbf{v}_D\right)\right) & =\left(A_S \mathbf{u}_S, \mathbf{v}_S\right)+\left(A_D \mathbf{u}_D, \mathbf{v}_D\right) \\
& \leq \max \{C_S, C_D\}\left\|\left(\mathbf{u}_S, \mathbf{u}_D\right)\right\|_{\mathbf{V}_S \times \mathbf{V}_D}\left\|\left(\mathbf{v}_S, \mathbf{v}_D\right)\right\|_{\mathbf{V}_S \times \mathbf{V}_D}.
\end{align*}
Similarly, the boundedness of operator $B$ is derived from \cref{b-continuity}  and the subsequent inequalities:
\begin{align*}
\left(B_S \mathbf{u}_S, q_S\right) & +\left(T_S \mathbf{u}_S, \phi\right)+\left(B_D \mathbf{u}_D, q_D\right)+\left(T_D \mathbf{u}_D, \phi\right) \\
& \leq D_S\left\|\mathbf{u}_S\right\|_{\mathbf{V}_S}\left(\left\|q_S\right\|_{Q_S}^2+\|\phi\|_{\Lambda_S}^2\right)^{\frac{1}{2}}+D_D\left\|\mathbf{u}_D\right\|_{\mathbf{V}_D}\left(\left\|q_D\right\|_{Q_D}^2+\|\phi\|_{\Lambda_D}^2\right)^{\frac{1}{2}} \\
& \leq \max \{D_S, D_D\}\left(\left\|\mathbf{u}_S\right\|_{\mathbf{V}_S}^2+\left\|\mathbf{u}_D\right\|_{\mathbf{V}_D}^2\right)^{\frac{1}{2}}\left(\left\|q_S\right\|_{Q_S}^2+\|\phi\|_{\Lambda_S}^2+\left\|q_D\right\|_{Q_D}^2+\|\phi\|_{\Lambda_D}^2\right)^{\frac{1}{2}} \\
& =\max \{D_S, D_D\}\left\|\left(\mathbf{u}_S, \mathbf{u}_D\right)\right\|_{\mathbf{V}_S \times \mathbf{V}_D}\left(\left\|\left(q_S, q_D\right)\right\|_{Q_S \times Q_D}^2+\|\phi\|_{\Lambda_S \cap \Lambda_D}^2\right)^{\frac{1}{2}},
\end{align*}
for all $\left(\mathbf{u}_S, \mathbf{u}_D, p_S, p_D, \lambda\right) \in \mathbf{V}_S \times \mathbf{V}_D \times Q_S \times Q_D \times\left(\Lambda_S \cap \Lambda_D\right)$, where the second inequality is due to the Cauchy-Schwarz inequality. Hence $A$ and $B$ are both bounded, with boundedness constants depending only on those of the subproblems.

For coercivity, note that because $B\left(\begin{array}{c} \mathbf{u}_S \\ \mathbf{u}_D\end{array}\right)=\left(\begin{array}{c}B_S \mathbf{u}_S \\ B_D \mathbf{u}_D \\ T_S \mathbf{u}_S + T_D \mathbf{u}_D\end{array}\right)$, we have that
\begin{equation*}
\operatorname{ker} B=\left(\operatorname{ker} B_S \times \operatorname{ker} B_D\right) \cap \operatorname{ker}\left(\begin{array}{cc}
T_S & T_D
\end{array}\right) 
\subset \operatorname{ker} B_S \times \operatorname{ker} B_D.
\end{equation*}
Based on the assumption of \eqref{a-coercivity-stronger}, 
for any $\left(\mathbf{u}_S, \mathbf{u}_D\right) \in \operatorname{ker} B$,
\begin{align*}
\left(A\left(\mathbf{u}_S, \mathbf{u}_D\right),\left(\mathbf{u}_S, \mathbf{u}_D\right)\right)_{\mathbf{V}_S \times \mathbf{V}_D} & =\left(A_S \mathbf{u}_S, \mathbf{u}_S\right)_{\mathbf{V}_S}+\left(A_D \mathbf{u}_D, \mathbf{u}_D\right)_{\mathbf{V}_D} \\
& \geq \min \{\alpha_S, \alpha_D\}\left\|\left(\mathbf{u}_S, \mathbf{u}_D\right)\right\|_{\mathbf{V}_S \times \mathbf{V}_D}^2.
\end{align*}
Thus $A$ is coercive on ker $B$ with constant $\min \{\alpha_S, \alpha_D\}$.

Finally, let us verify the inf-sup condition \eqref{inf-sup}. Let $R_{Q_i}^{-1}: Q_i \rightarrow Q_i^{\prime}, R_{\Lambda_i}^{-1}: \Lambda_i \rightarrow \Lambda_i^{\prime}$ be the inverse Riesz maps of their corresponding spaces. By the Riesz Representation Theorem, this is an isometry between $Q_i$ and $Q_i^{\prime}$, meaning that $( R_{Q_i}^{-1} q, q )=\|q\|_{Q_i}^2$.  Given $\left(q_S, q_D, \phi \right) \in Q_S \times Q_D \times (\Lambda_S\cap \Lambda_D)$, let $\mathbf{u}_i^*, p_i^*, \lambda_i^*$ be the solution of
\begin{equation*}
\left(\begin{array}{ccc}
A_i & B_i^{\prime} & T_i^{\prime} \\
B_i & & \\
T_i & &
\end{array}\right)\left(\begin{array}{c}
\mathbf{u}_i^* \\
p_i^* \\
\lambda_i^*
\end{array}\right)=\left(\begin{array}{c}
\mathbf{0} \\
R_{Q_i}^{-1} q_i \\
R_{\Lambda_i}^{-1} \phi
\end{array}\right) \quad \text { for } i=S,D.
\end{equation*}
Considering $\mathbf{u}_i^* = \mathbf{u}_i^*\left(q_i, \phi\right)$ as a function of $q_i, \phi$, by \eqref{eq:estimate-abstract} we have that
\begin{equation*}
\|\mathbf{u}_i^*\|_{\mathbf{V}_i}^2 \leq 2 E_i^2\left(\|R_{Q_i}^{-1} q_i\|_{Q_{i}^{\prime}}^2+\|R_{\Lambda_i}^{-1} \phi\|_{\Lambda_{i}^{\prime}}^2\right)=2 E_i^2\left(\left\|q_i\right\|_{Q_i}^2+\|\phi\|_{\Lambda_i}^2\right),
\end{equation*}
for any $\left(q_i, \phi\right) \in Q_i \times \Lambda_i$.
Furthermore, we have that $\left(B_i \mathbf{u}_i^*, q_i\right)+\left(T_i \mathbf{u}_i^*, \phi\right)=(R_{Q_i}^{-1} q_i, q_i)+(R_{\Lambda_i}^{-1} \phi, \phi)=\left\|q_i\right\|_{Q_i}^2+$ $\|\phi\|_{\Lambda_i}^2$. Then, there holds
\begin{align*}
& \sup_{\left(\mathbf{u}_S, \mathbf{u}_D\right) \in \mathbf{V}_S \times \mathbf{V}_D} \frac{\left(B_S \mathbf{u}_S, q_S\right)+\left(B_D \mathbf{u}_D, q_D\right)+\left(T_S\mathbf{u}_S, \phi\right)+\left(T_D \mathbf{u}_D, \phi\right)}{\left(\left\|\mathbf{u}_S\right\|_{\mathbf{V}_S}^2+\left\|\mathbf{u}_D\right\|_{\mathbf{V}_D}^2\right)^{1/2}} \\
 & \quad\quad \geq \frac{\left(B_S \mathbf{u}_S^*, q_S\right)+\left(B_D \mathbf{u}_D^*, q_D\right)+\left(T_S \mathbf{u}_S^*, \phi\right)+\left(T_D \mathbf{u}_D^*, \phi\right)}{\left(\left\|\mathbf{u}_S^*\right\|_{\mathbf{V}_S}^2+\left\|\mathbf{u}_D^*\right\|_{\mathbf{V}_D}^2\right)^{1 / 2}} \\
 & \quad\quad \geq  \frac{1}{E} \frac{\left\|q_S\right\|_{Q_S}^2+\left\|q_D\right\|_{Q_D}^2+\|\phi\|_{\Lambda_S}^2+\|\phi\|_{\Lambda_D}^2}{\left(\left\|q_S\right\|_{Q_S}^2+\left\|q_D\right\|_{Q_D}^2+\|\phi\|_{\Lambda_S}^2+\|\phi\|_{\Lambda_D}^2\right)^{1 / 2}} \\
&\quad\quad =  \frac{1}{E}\left(\left\|q_S\right\|_{Q_S}^2+\left\|q_D\right\|_{Q_D}^2+\|\phi\|_{\Lambda_S \cap \Lambda_D}^2\right)^{1 / 2},
\end{align*}
where $E=2 \max \{E_S, E_D\}^2$. Hence we obtain the desired inf-sup condition.
\end{proof}

\section{Well-posedness and preconditioning}\label{sec3:well-posed}
In this section, we establish the well-posedness of the coupled Stokes-Darcy problem within the abstract framework of \Cref{thm-abstract}. Our analysis will consider different arrangements of boundary conditions in the coupled problems, labeled as NN, EE, NE(NE$^{*}$), and EN(EN$^{*}$) according to the boundary conditions intersecting the interface, as shown in \cref{fig:bc_config}.  For example, NN implies that $\Gamma_S^N\cap\Gamma \neq \emptyset$ (and $\Gamma_S^E\cap\Gamma=\emptyset$) and $\Gamma_D^N\cap\Gamma \neq \emptyset$ (and $\Gamma_D^E\cap\Gamma=\emptyset$) with the other cases following a similar pattern. This distinction is necessary because the Lagrange multiplier space will vary accordingly.


\begin{figure}
  \setlength{\abovecaptionskip}{0.0625\baselineskip}    
  \setlength{\belowcaptionskip}{0.0625\baselineskip}  
    \centering
\begin{subfigure}{0.45\textwidth}
    \centering
\begin{tikzpicture}[line width=1pt,scale=2]
\draw[black, ultra thick] (0,0) -- (2,0);
\draw[black, ultra thick] (0,0) -- (0,1);
\draw[gray, ultra thick] (1,0) -- (1,1);
\draw[black, ultra thick] (0,1) -- (2,1);
\draw[black, ultra thick] (2,0) -- (2,1);
\filldraw[black] (0.4,-0.15) node[anchor=west]{$\Gamma_S^N$};
\filldraw[black] (0.4,1.15) node[anchor=west]{$\Gamma_S^N$};
\filldraw[black] (-0.4,0.5) node[anchor=west]{$\Gamma_S^E$};
\filldraw[black] (1.4,-0.15) node[anchor=west]{$\Gamma_D^N$};
\filldraw[black] (1.4,1.15) node[anchor=west]{$\Gamma_D^N$};
\filldraw[black] (2.02,0.5) node[anchor=west]{$\Gamma_D^E$};
\filldraw[black] (0.4,0.5) node[anchor=west]{$\Omega_S$};
\filldraw[black] (1.4,0.5) node[anchor=west]{$\Omega_D$};
\end{tikzpicture}
\caption{NN}\label{NN}
\end{subfigure}
\hfill
\begin{subfigure}{0.45\textwidth}
    \centering
\begin{tikzpicture}[line width=1pt,scale=2]
\draw[black, ultra thick] (0,0) -- (2,0);
\draw[black, ultra thick] (0,0) -- (0,1);
\draw[gray, ultra thick] (1,0) -- (1,1);
\draw[black, ultra thick] (0,1) -- (2,1);
\draw[black, ultra thick] (2,0) -- (2,1);
\filldraw[black] (0.4,-0.15) node[anchor=west]{$\Gamma_S^E$};
\filldraw[black] (0.4,1.15) node[anchor=west]{$\Gamma_S^E$};
\filldraw[black] (-0.4,0.5) node[anchor=west]{$\Gamma_S^E$};
\filldraw[black] (1.4,-0.15) node[anchor=west]{$\Gamma_D^E$};
\filldraw[black] (1.4,1.15) node[anchor=west]{$\Gamma_D^E$};
\filldraw[black] (2.02,0.5) node[anchor=west]{$\Gamma_D^E$};
\filldraw[black] (0.4,0.5) node[anchor=west]{$\Omega_S$};
\filldraw[black] (1.4,0.5) node[anchor=west]{$\Omega_D$};
\end{tikzpicture}
\caption{EE}\label{DD}
\end{subfigure}
\begin{subfigure}{0.45\textwidth}
    \centering
\begin{tikzpicture}[line width=1pt,scale=2]
\draw[black, ultra thick] (0,0) -- (2,0);
\draw[black, ultra thick] (0,0) -- (0,1);
\draw[gray, ultra thick] (1,0) -- (1,1);
\draw[black, ultra thick] (0,1) -- (2,1);
\draw[black, ultra thick] (2,0) -- (2,1);
\filldraw[black] (0.4,-0.15) node[anchor=west]{$\Gamma_S^N$};
\filldraw[black] (0.4,1.15) node[anchor=west]{$\Gamma_S^N$};
\filldraw[black] (-0.4,0.5) node[anchor=west]{$\Gamma_S^E$};
\filldraw[black] (1.4,-0.15) node[anchor=west]{$\Gamma_D^E$};
\filldraw[black] (1.4,1.15) node[anchor=west]{$\Gamma_D^E$};
\filldraw[black] (2.02,0.5) node[anchor=west, color=black]{$\Gamma_D^N$};
\filldraw[black] (0.4,0.5) node[anchor=west]{$\Omega_S$};
\filldraw[black] (1.4,0.5) node[anchor=west]{$\Omega_D$};
\end{tikzpicture}
\caption{NE$^*$}\label{NDnoker}
\end{subfigure}
\hfill
\begin{subfigure}{0.45\textwidth}
    \centering
\begin{tikzpicture}[line width=1pt,scale=2]
\draw[black, ultra thick] (0,0) -- (2,0);
\draw[black, ultra thick] (0,0) -- (0,1);
\draw[gray, ultra thick] (1,0) -- (1,1);
\draw[black, ultra thick] (0,1) -- (2,1);
\draw[black, ultra thick] (2,0) -- (2,1);
\filldraw[black] (0.4,-0.15) node[anchor=west]{$\Gamma_S^E$};
\filldraw[black] (0.4,1.15) node[anchor=west]{$\Gamma_S^E$};
\filldraw[black] (-0.4,0.5) node[anchor=west, color=black]{$\Gamma_S^N$};
\filldraw[black] (1.4,-0.15) node[anchor=west]{$\Gamma_D^N$};
\filldraw[black] (1.4,1.15) node[anchor=west]{$\Gamma_D^N$};
\filldraw[black] (2.02,0.5) node[anchor=west]{$\Gamma_D^E$};
\filldraw[black] (0.4,0.5) node[anchor=west]{$\Omega_S$};
\filldraw[black] (1.4,0.5) node[anchor=west]{$\Omega_D$};
\end{tikzpicture}
\caption{EN$^*$}\label{DNnoker}
\end{subfigure}
\begin{subfigure}{0.45\textwidth}
    \centering
\begin{tikzpicture}[line width=1pt,scale=2]
\draw[black, ultra thick] (0,0) -- (2,0);
\draw[black, ultra thick] (0,0) -- (0,1);
\draw[gray, ultra thick] (1,0) -- (1,1);
\draw[black, ultra thick] (0,1) -- (2,1);
\draw[black, ultra thick] (2,0) -- (2,1);
\filldraw[black] (0.4,-0.15) node[anchor=west]{$\Gamma_S^N$};
\filldraw[black] (0.4,1.15) node[anchor=west]{$\Gamma_S^N$};
\filldraw[black] (-0.4,0.5) node[anchor=west]{$\Gamma_S^E$};
\filldraw[black] (1.4,-0.15) node[anchor=west]{$\Gamma_D^E$};
\filldraw[black] (1.4,1.15) node[anchor=west]{$\Gamma_D^E$};
\filldraw[black] (2.02,0.5) node[anchor=west, color=red]{$\Gamma_D^E$};
\filldraw[black] (0.4,0.5) node[anchor=west]{$\Omega_S$};
\filldraw[black] (1.4,0.5) node[anchor=west]{$\Omega_D$};
\end{tikzpicture}
\caption{NE}\label{ND}
\end{subfigure}
\hfill
\begin{subfigure}{0.45\textwidth}
    \centering
\begin{tikzpicture}[line width=1pt,scale=2]
\draw[black, ultra thick] (0,0) -- (2,0);
\draw[black, ultra thick] (0,0) -- (0,1);
\draw[gray, ultra thick] (1,0) -- (1,1);
\draw[black, ultra thick] (0,1) -- (2,1);
\draw[black, ultra thick] (2,0) -- (2,1);
\filldraw[black] (0.4,-0.15) node[anchor=west]{$\Gamma_S^E$};
\filldraw[black] (0.4,1.15) node[anchor=west]{$\Gamma_S^E$};
\filldraw[black] (-0.4,0.5) node[anchor=west, color=red]{$\Gamma_S^E$};
\filldraw[black] (1.4,-0.15) node[anchor=west]{$\Gamma_D^N$};
\filldraw[black] (1.4,1.15) node[anchor=west]{$\Gamma_D^N$};
\filldraw[black] (2.02,0.5) node[anchor=west]{$\Gamma_D^E$};
\filldraw[black] (0.4,0.5) node[anchor=west]{$\Omega_S$};
\filldraw[black] (1.4,0.5) node[anchor=west]{$\Omega_D$};
\end{tikzpicture}
\caption{EN}\label{DN}
\end{subfigure}
\caption{Different configurations of boundary conditions considered with the coupled Stokes-Darcy problem. (a)-(d) do not lead to near kernel. Highlighted in color is the change of boundary conditions in (c) and (d), which yields a near kernel in cases (e) and (f).
}
\label{fig:bc_config}
\end{figure}
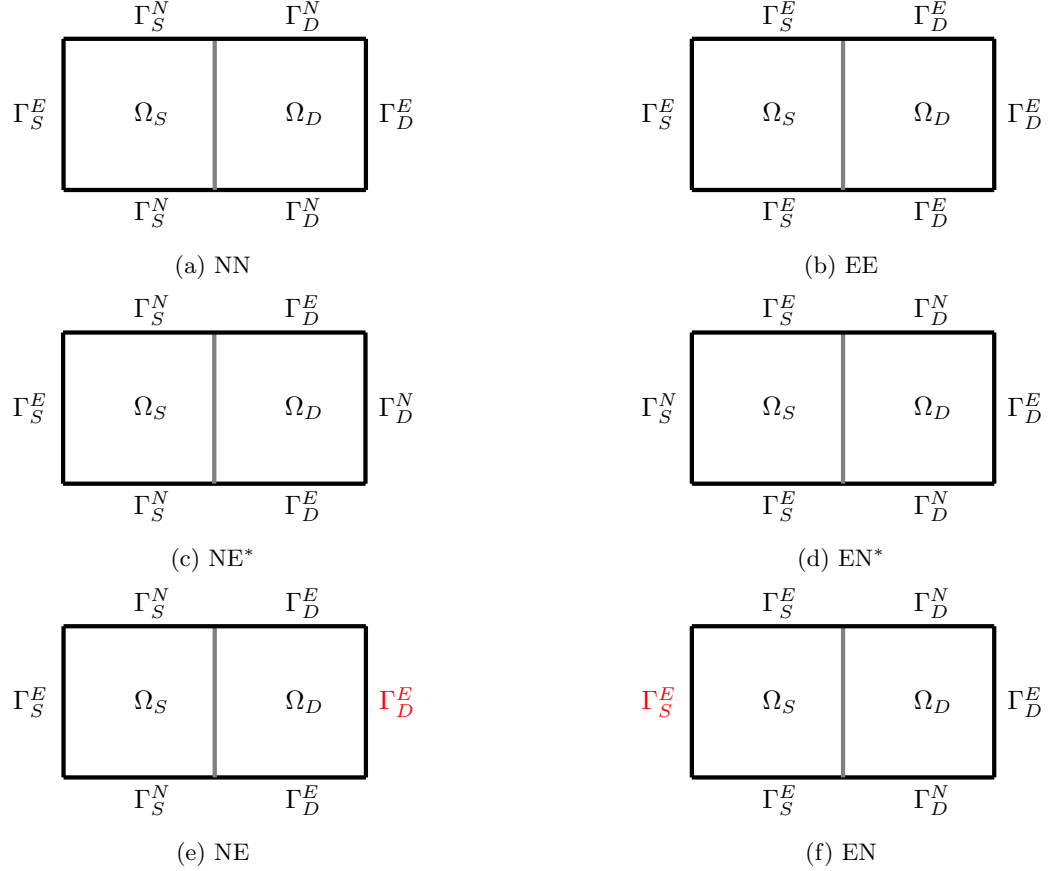

\subsection{Well-posedness of the Stokes-Darcy problem}
Following \Cref{thm-abstract} the well-posedness of the coupled problem requires that the Stokes and Darcy subproblems are individually well-posed, which are subjects of the following two lemmas.

\begin{lemma}[Stokes subproblem with mixed boundary conditions]\label{lemma-Stokes} Let $\Omega_S \subset \mathbb{R}^d$ be a bounded domain with boundary decomposition $\partial \Omega_S=\Gamma \cup  \Gamma_S^E \cup \Gamma_S^N$, where the components are assumed to be of non-zero measure and $\Gamma_S^E\cap \Gamma = \emptyset$. We consider the Stokes problem \eqref{Stokes-model} with the boundary conditions  $\mathbf{u}_S|_{\Gamma_S^E} = \mathbf{0}, \ \bm{\sigma}(\mathbf{u}_S,p_S)\cdot\mathbf{n}_S|_{\Gamma_S^N} = \mathbf{0}$ and
\begin{align}
\mathbf{u}_S \cdot \mathbf{n}_S = g_n, & \quad \text { on } \Gamma, \label{Stokes-inter-1}\\
\boldsymbol{\tau} \cdot\left(2\mu \bm{\epsilon}(\mathbf{u}_S)-p_S I\right) \cdot \mathbf{n}_S +  {\beta_{\tau}}\mathbf{u}_S \cdot \boldsymbol{\tau}  = 0, & \quad \text { on } \Gamma, \label{Stokes-inter-2}
\end{align}
where \eqref{Stokes-inter-1} shall be enforced by a Lagrange multiplier. Let the Lagrange multiplier space be $\Lambda_S = \frac{1}{\sqrt{\mu}} H^{-1 / 2}(\Gamma)$ equipped with the norm $
    \|\lambda\|_{\Lambda_S} = \frac{1}{\sqrt{\mu}}\|\lambda\|_{-\frac{1}{2},\Gamma}. $
Use the velocity and pressure spaces defined in \eqref{Stokes-spaces} and equip $\mathbf{W}=\mathbf{V}_S \times Q_S \times \Lambda_S$ with the norm
\begin{align}\label{norm-stokes}
   & \|(\mathbf{u}_S,p_S,\lambda)\|_{\mathbf{W}} = \left(\|\mathbf{u}_S\|_{\mathbf{V}_S}^2 + \|p_S\|_{Q_S}^2 + \|\lambda\|_{\Lambda_S}^2\right)^{1/2}.
\end{align}
Let $\bm{f}_S \in \mathbf{V}_S^{\prime}$, $g_n \in \Lambda_S^{\prime}$, and  define the bilinear form $a_S: \mathbf{W} \times \mathbf{W} \rightarrow \mathbb{R}$ and a linear operator $L: \mathbf{W} \rightarrow \mathbb{R}$ as follows,
\begin{align*}
a_S((\mathbf{u}_S, p_S, \lambda),(\mathbf{v}_S, q_S, \phi))& = 2\mu(\bm{\epsilon}(\mathbf{u}_S), \bm{\epsilon}(\mathbf{v}_S))_S + \beta_{\tau}(T_t \mathbf{u}_S, T_t \mathbf{v}_S)_{\Gamma} - (\nabla \cdot \mathbf{v}_S, p_S)_S  \\
&\quad -(\nabla \cdot \mathbf{u}_S, q_S)_S + (T_S \mathbf{v}_S, \lambda)_{\Gamma} + ( T_S \mathbf{u}_S, \phi)_{\Gamma},  \\
L((\mathbf{v}_S, q_S, \phi)) & = (\bm{f}_S, \mathbf{v}_S)_S + (g_n, \phi)_{\Gamma}.
\end{align*}
Note that the multiplier space reflects that in our setting $\Gamma$ intersects only with $\Gamma^N_S$. Then the Stokes problem: find $\left(\mathbf{u}_S, p_S, \lambda\right) \in \mathbf{W}$ such that
\begin{equation}\label{Stokes-subproblem}
a_S((\mathbf{u}_S, p_S, \lambda),(\mathbf{v}_S, q_S, \phi))=L((\mathbf{v}_S, q_S, \phi)), \quad \forall\,(\mathbf{v}_S, q_S, \phi) \in \mathbf{W},
\end{equation}
is well-posed. In particular, the Brezzi conditions ensure the following stability estimate
$$
\left\|\left(\mathbf{u}_S, p_S, \lambda\right)\right\|_{\mathbf{W}} \leq C\left(\left\|\bm{f}_S\right\|_{\mathbf{V}_S^{\prime}} + \left\|g_n\right\|_{\Lambda_S^{\prime}}\right),
$$
with $C$ independent of the parameters $\mu$ and $\beta_{\tau}$.
\end{lemma}
\begin{proof}
The operator form of problem \eqref{Stokes-subproblem} is
\begin{equation*}
\left(\begin{array}{ccc}
A_S & B_S^{\prime} & T_S^{\prime} \\
B_S & & \\
T_S & &
\end{array}\right)\left(\begin{array}{c}
\mathbf{u}_S \\
p_S \\
\lambda
\end{array}\right)=\left(\begin{array}{c}
\mathbf{f}_S \\
0 \\
g_n
\end{array}\right)
\end{equation*}
In order to prove that the problem \eqref{Stokes-subproblem} is well-posed, we need to verify the conditions \eqref{a-continuity}-\eqref{b-continuity} and \eqref{a-coercivity-stronger}.  The continuity \eqref{a-continuity} and coercivity \eqref{a-coercivity-stronger} are promptly satisfied.  The boundedness \eqref{b-continuity} follows from the Cauchy-Schwarz inequality and trace inequality,
\begin{align*}
     (B_S\mathbf{u}_S, q_S) + (T_S\mathbf{u}_S, \phi)
    & = - (\nabla\cdot\mathbf{u}_S, q_S)_S + (\mathbf{u}_S\cdot\mathbf{n}_S, \phi)_{\Gamma} \\
    & \leq \|\nabla \cdot \mathbf{u}_S\|_{S}\left\|q_S\right\|_{S} + \| \mathbf{u}_S \cdot \mathbf{n}_S \|_{\frac{1}{2}, \Gamma} \|\phi\|_{-\frac{1}{2}, \Gamma} \\
     & \lesssim\|\bm{\epsilon}(\mathbf{u}_S)\|_{S}\left(\|q_S\|_{S}+\|\phi\|_{-\frac{1}{2}, \Gamma}\right) \\
     & \lesssim \|\mathbf{u}_S\|_{\mathbf{V}_S} \left(\|q_S\|_{Q_S}^2 + \|\phi\|_{\Lambda_S}^2\right)^{1/2}.
\end{align*}

Next we verify the inf-sup condition \eqref{inf-sup}, that is,
\begin{equation}\label{inf-sup-Stokes}
    \sup_{\mathbf{v}_S\in\mathbf{V}_S\backslash\{0\}}\frac{(B_S\mathbf{v}_S,p_S)+(T_S\mathbf{v}_S,\lambda)}{\|\mathbf{v}_S\|_{\mathbf{V}_S}} \geq \gamma_S \left(\|p_S\|_{Q_S}^2 + \|\lambda\|_{\Lambda_S}^2\right)^{1/2}.
\end{equation}
Firstly, construct $\mathbf{v}^{p_S}\in\mathbf{V}_S$ satisfying the Stokes inf-sup condition 
\begin{subequations}
\begin{align}
    & -\nabla\cdot\mathbf{v}^{p_S} = p_S, \quad \mathbf{v}^{p_S}|_{\Gamma} = 0, \label{vps-1}\\
    & \|\bm{\epsilon}(\mathbf{v}^{p_S})\|_{S} \lesssim \|p_S\|_{S}.\label{vps-2} 
\end{align}
\end{subequations}
Then, for given $\lambda\in H^{-1/2}(\Gamma)$, let $\psi\in H^{1/2}(\Gamma)$ be the Riesz representation of $\lambda$ with respect to $(\cdot, \cdot)_{\Gamma}$. Define $\mathbf{v}^{\lambda}\in \mathbf{H}^{1}(\Omega_S)$ such that
\begin{subequations}
\begin{align}
    & \nabla\cdot\mathbf{v}^{\lambda} = 0, \quad \mathbf{v}^{\lambda}|_{\Gamma} = \psi\mathbf{n}_S, \label{vl-1}\\
    & \|\bm{\epsilon}(\mathbf{v}^{\lambda})\|_{S} \lesssim \|\psi\|_{\frac{1}{2},\Gamma} = \|\lambda\|_{-\frac{1}{2},\Gamma}. \label{vl-2}
\end{align}
\end{subequations}
Let us choose the test function $\mathbf{v}_S = \mu^{-1}(\mathbf{v}^{p_S}+\mathbf{v}^{\lambda})$, combine \eqref{vps-1} with \eqref{vl-1},  so that, since $\mathbf{v}_S\cdot\bm{\tau} = 0$, we obtain 
\begin{align}
    (B_S\mathbf{v}_S,p_S) + (T_S\mathbf{v}_S,\lambda) & = -(\nabla\cdot\mathbf{v}_S,p_S)_S + (\mathbf{v}_S\cdot\mathbf{n}_S, \lambda)_{\Gamma} \nonumber \\
    & = -\mu^{-1}(\nabla\cdot\mathbf{v}^{p_S},p_S)_S + \mu^{-1} (\mathbf{v}^{\lambda}\cdot\mathbf{n}_S, \lambda)_{\Gamma} = \mu^{-1}\|p_S\|_{S}^2 + \mu^{-1}\|\lambda\|_{-\frac{1}{2},\Gamma}^2. \label{estimate-1}
\end{align}
Furthermore, note that $T_t\mathbf{v}_S = \mathbf{v}_S\cdot\bm{\tau} = 0$, the following estimate follows directly from the estimates \eqref{vps-2} and \eqref{vl-2},
\begin{align}
 \|\mathbf{v}_S\|_{\mathbf{V}_S}^2 & = \mu \|\bm{\epsilon}(\mathbf{v}_S)\|_{S}^2 + \beta_{\tau}\|T_t\mathbf{v}_S\|_{\Gamma}^2 \nonumber \\
    & \leq \mu^{-1} \| \bm{\epsilon}(\mathbf{v}^{p_S})\|_{S}^2 + \mu^{-1} \| \bm{\epsilon}(\mathbf{v}^{\lambda})\|_{S}^2  \lesssim \mu^{-1}\|p_S\|_{S}^2 + \mu^{-1}\|\lambda\|_{-\frac{1}{2},\Gamma}^2. \label{estimate-2}
\end{align}
Combining \eqref{estimate-1} with \eqref{estimate-2} leads to the inf-sup condition \cref{inf-sup-Stokes}. Based on the Brezzi theory, the Stokes subproblem is well-posed with respect to the norm \eqref{norm-stokes}.  
\end{proof}

\begin{remark}[Stokes subproblem with pure Dirichlet conditions] \label{remark-Stokes}
 Assume that $\partial\Omega_S = \Gamma\cup\Gamma_S^E$, i.e. $\Gamma_S^N = \emptyset$. Consequently, the pressure and Lagrange multiplier spaces are adjusted to $(Q_S\times\Lambda_S)/\mathbb{R}=(\frac{1}{\sqrt{\mu}}L^2(\Omega_S)\times \frac{1}{\sqrt{\mu}}H^{-1/2}_{00}(\Gamma))/\mathbb{R}$. In this case, we still construct the Stokes problem \eqref{Stokes-model} with the boundary conditions $\mathbf{u}_S|_{\Gamma_S^E} = 0$, and \eqref{Stokes-inter-1}-\eqref{Stokes-inter-2} on $\Gamma$. Using the same velocity space $\mathbf{V}_S$ and repeating the same proof process outlined above, we obtain the same results as established in \cref{lemma-Stokes}.
\end{remark}

\begin{lemma}[Darcy subproblem with mixed boundary conditions]\label{lemma-Darcy}
Let $\Omega_D \subset \mathbb{R}^d$ be a bounded domain with boundary decomposition $\partial \Omega_D=\Gamma \cup \Gamma_D^E\cup \Gamma_D^N$, where the components are assumed to be of non-zero measure and $\Gamma_D^E\cap\Gamma=\emptyset$. We consider the Darcy problem \eqref{Darcy-model} with the boundary condition $\mathbf{u}_D\cdot\mathbf{n}_D|_{\Gamma_D^E} = 0, \ p_D|_{\Gamma_D^N} = 0$ and 
\begin{align}\label{Darcy-inter}
\mathbf{u}_D \cdot \mathbf{n}_D = -g_n, & \quad \text { on } \Gamma,
\end{align}
which shall be enforced by a Lagrange multiplier. Let the Lagrange multiplier space be $\Lambda_D = \sqrt{{K}} H_{00}^{1/2}(\Gamma)$ equipped with the weighted norm $\|\lambda\|_{\Lambda_D} = \sqrt{{K}}\|\lambda\|_{\frac{1}{2},\Gamma}$. Use the velocity and pressure spaces \eqref{Darcy-spaces} and equip $\mathbf{W}=\mathbf{V}_D \times Q_D \times \Lambda_D$ with the norm
\begin{align*}
   & \|(\mathbf{u}_D,p_D,\lambda)\|_{\mathbf{W}} = \left(\|\mathbf{u}_D\|_{\mathbf{V}_D}^2 + \|p_D\|_{Q_D}^2 + \|\lambda\|_{\Lambda_D}^2\right)^{1/2}.
\end{align*}
Let $g_D \in Q_D^{\prime}$, $g_n \in \Lambda_D^{\prime}$, and
define the bilinear form $a_D : \mathbf{W} \times \mathbf{W} \to \mathbb{R}$  and a linear form $L: \mathbf{W} \to \mathbb{R}$ as follows, 
\begin{align*}
a_D((\mathbf{u}_D, p_D, \lambda),(\mathbf{v}_D, q_D, \phi))& = {K}^{-1}(\mathbf{u}_D, \mathbf{v}_D)_D - ( \nabla \cdot \mathbf{v}_D, p_D)_D \\
& \quad +(\nabla \cdot \mathbf{u}_D, q_D)_D + (T_D \mathbf{v}_D, \lambda)_{\Gamma} + (T_D \mathbf{u}_D, \phi)_{\Gamma},  \\
L((\mathbf{v}_D, q_D, \phi))& =(g_D, q_D)_D - (g_n, \phi)_{\Gamma} .
\end{align*}
Then the Darcy problem: find $(\mathbf{u}_D, p_D, \lambda) \in \mathbf{W}$ such that
\begin{equation*}
a_D((\mathbf{u}_D, p_D, \lambda),(\mathbf{v}_D, q_D, \phi))=L((\mathbf{v}_D, q_D, \phi)), \quad \forall\,(\mathbf{v}_D, q_D, \phi) \in \mathbf{W},
\end{equation*}
is well-posed. In particular, the Brezzi conditions ensure the following stability estimate
\begin{equation*}
\|(\mathbf{u}_D, p_D, \lambda)\|_\mathbf{W} \leq C\left(\|g_D\|_{Q_D^{\prime}} + \|g_n\|_{\Lambda_D^{\prime}}\right),
\end{equation*}
with $C$ independent of the parameter $K$.
\end{lemma}

\begin{proof}
In this lemma, we mainly establish the Darcy inf-sup condition on domain $\Omega_D$ for which we follow the same approach as in \cref{lemma-Stokes}.  
For given $(p_D,\lambda)$ in $Q_D\times \Lambda_D$, we construct $w\in H^{1}(\Omega_D)$ as the solution of the following elliptic problem
  \begin{subequations}\label{construct-vd}
      \begin{align}
         -\nabla\cdot\nabla w = {K} p_D, & \quad \text{ in } \Omega_D, \label{construct-vd-1}\\
        \nabla w\cdot\mathbf{n}_D = {K} \mathcal{D}_{\Gamma}  \lambda, & \quad  \text{ on } \Gamma, \label{construct-vd-2}\\
          \nabla w\cdot\mathbf{n}_D  = 0, & \quad  \text{ on } \Gamma_D^E, \label{construct-vd-3}\\
          w = 0, & \quad  \text{ on } \Gamma_D^N, \label{construct-vd-4}
      \end{align}
  \end{subequations}
  where $\mathcal{D}_{\Gamma}: H^{1/2}_{00}(\Gamma)\rightarrow H^{-1/2}_{00}(\Gamma)$ is the Dirichlet-to-Neumann (DtN) mapping~\cite{sylvester1990dirichlet}. It is defined by $\mathcal{D}_{\Gamma}\lambda = \nabla u \cdot \mathbf{n}_D|_{\Gamma}$, where $u$ is the solution of 
\begin{equation}\label{aux-laplace}
\begin{cases}
    -\Delta u = 0, \quad \text{in} \ \Omega_D, \\
     u|_{\Gamma}= \lambda, \quad  u|_{\Gamma_D^N \cup \Gamma_D^E} = 0.  \\
\end{cases}
\end{equation}

Integrating by parts for \eqref{construct-vd-1}, invoking Cauchy-Schwarz inequality, the boundness-property of the DtN mapping, the trace inequality, and Poincar\'{e} inequality, we have
  \begin{align*}
      \|\nabla w\|_{D}^2 & = (\nabla w\cdot \mathbf{n}_D,w)_{\Gamma} + {K} (p_D,w)_D \leq  {K}\| \mathcal{D}_{\Gamma} \lambda\|_{-\frac{1}{2},\Gamma}\|w\|_{\frac{1}{2},\Gamma}  + {K}\|p_D\|_{D} \|w\|_{D}\\
      & \lesssim {K}\|\lambda\|_{\frac{1}{2},\Gamma}\|w\|_{\frac{1}{2},\Gamma}  + {K}\|p_D\|_{D} \|w\|_{D} \lesssim {K} (\|\lambda\|_{\frac{1}{2},\Gamma}^2 + \|p_D\|_{D}^2)^{1/2} \|\nabla w\|_{D}.
  \end{align*}
Letting $\mathbf{v}_D = \nabla w$ we have
\begin{align}\label{bound-vd}
    {K}^{-1}\|\mathbf{v}_D\|_{D}^2 \lesssim {K}( \|p_D\|_D^2 + \|\lambda\|_{\frac{1}{2},\Gamma}^2),
\end{align}
and from the \eqref{construct-vd-1}, there holds
  \begin{align}\label{bound-div-vd}
      & {K}^{-1}\|\nabla\cdot\mathbf{v}_D\|_D^2 = {K} \|p_D\|_{D}^2.
  \end{align}
  Combining \eqref{bound-div-vd} with \eqref{bound-vd} leads to the following estimate, 
  \begin{align}\label{vd-bound}
      \|\mathbf{v}_D\|_{\mathbf{V}_D} & \lesssim  \|p_D\|_{Q_D} + \|\lambda \|_{\Lambda_D},
  \end{align}
 In addition, combining \eqref{construct-vd-2} with \eqref{construct-vd-3}, we have
  \begin{align}\label{vd-inf-sup}
      -(\nabla\cdot\mathbf{v}_D,p_D)_D + (\mathbf{v}_D\cdot\mathbf{n}_D,\lambda)_{\Gamma} = {K}\|p_D\|_{D}^2 +  {K}\|\lambda\|_{\frac{1}{2},\Gamma}^2. 
  \end{align}
  Then, the inf-sup condition \cref{inf-sup} follows from \eqref{vd-bound} and \eqref{vd-inf-sup}. Combining with he continuity and coercivity conditions completes the proof. 
\end{proof}

\begin{remark}[Darcy subproblem with pure Dirichlet conditions]\label{remark-Darcy}
  Assume that $\partial\Omega_D = \Gamma\cup\Gamma_D^E$, i.e. $\Gamma_D^N = \emptyset$. We still construct the Darcy problem \eqref{Darcy-model} with the boundary condition $\mathbf{u}_D\cdot\mathbf{n}_D|_{\Gamma_D^E} = 0$, and the condition \eqref{Darcy-inter} on $\Gamma$. The corresponding pressure and Lagrange multiplier spaces should be $(Q_D\times\Lambda)/\mathbb{R} = (\sqrt{{K}}L^2(\Omega_D)\times \sqrt{{K}}H^{1/2}(\Gamma))/\mathbb{R}$, and we use the same velocity space $\mathbf{V}_D$ as in \eqref{Darcy-spaces}. For the inf-sup condition, we construct the following auxiliary problem similar to \eqref{construct-vd}, 
    \begin{equation*}
      \begin{cases}
         -\nabla\cdot\nabla w = {K} p_D, & \text{ in } \Omega_D,\\
        \nabla w\cdot\mathbf{n}_D = {K} \mathcal{D}_{\Gamma} \lambda, & \text{ on } \Gamma,\\
          \nabla w\cdot\mathbf{n}_D  = 0, & \text{ on } \Gamma_D^E, 
      \end{cases}
    \end{equation*}
    which has a unique solution in {$H^1(\Omega_D)/\mathbb{R}$}. Here, $\mathcal{D}_{\Gamma}: H^{1/2}(\Gamma)\rightarrow H^{-1/2}(\Gamma)$, the difference from \eqref{aux-laplace} is that solution $u$ satisfies  $\nabla u\cdot\mathbf{n}_D|_{\Gamma_D^E} = 0$. Following
    the proof of \cref{lemma-Darcy}, we can obtain well-posedness of the Darcy problem.
\end{remark}

Now we are ready to show the well-posedness of the Stokes-Darcy problem \cref{weak-form}. From the well-posedness of each subproblem in \cref{lemma-Stokes} and \cref{lemma-Darcy}, \cref{remark-Stokes} and \cref{remark-Darcy}, we can formulate the following theorem ensuring the well-posedness of the coupled system.
\begin{theorem}\label{Thm-wellposedness}
  The Stokes-Darcy operator $\mathcal{A}$ in \eqref{operator-A} is an isomorphism mapping $\mathbf{W}$ to $\mathbf{W}^{\prime}$ such that $\|\mathcal{A}\|_{\mathcal{L}\left(\mathbf{W}, \mathbf{W}^{\prime}\right)} \leq C_1$ and $\left\|\mathcal{A}^{-1}\right\|_{\mathcal{L}\left(\mathbf{W}^{\prime}, \mathbf{W}\right)} \leq$ $C_2^{-1}$, where $C_1$ and $C_2$ are independent of the physical parameters. Here the Lagrange multiplier space $\Lambda$ and the solution space $\mathbf{W}$ reflect configuration of the boundary conditions and are defined as follows 
  \begin{align*}
    &\mbox{(NN)}\,\,
&\Lambda  =  \frac{1}{\sqrt{\mu}}H^{-1/2}(\Gamma)\cap \sqrt{{K}}H_{00}^{1/2}(\Gamma),\,
 &\mathbf{W}  = \mathbf{V}\times\frac{1}{\sqrt{\mu}} L^2\left(\Omega_S\right)\times\sqrt{{K}} L^2\left(\Omega_D\right)\times \Lambda, \\
&\mbox{(EE)}\,\,
    &\Lambda  = \frac{1}{\sqrt{\mu}}H_{00}^{-1/2}(\Gamma)\cap \sqrt{{K}}H^{1/2}(\Gamma),\,
   &\mathbf{W}  = \mathbf{V}\times\Big(\frac{1}{\sqrt{\mu}} L^2\left(\Omega_S\right) \times \sqrt{{K}} L^2\left(\Omega_D\right)\times\Lambda\Big)/\mathbb{R}, \\
&\mbox{(NE)}\,\,
&\Lambda  =  \frac{1}{\sqrt{\mu}}H^{-1/2}(\Gamma)\cap \sqrt{{K}}H^{1/2}(\Gamma),\, 
&\mathbf{W} = \mathbf{V}\times\frac{1}{\sqrt{\mu}} L^2\left(\Omega_S\right)\times\Big(\sqrt{{K}} L^2\left(\Omega_D\right)\times \Lambda\Big)/\mathbb{R}, \\
&\mbox{(NE$^{*}$)}\,\,
&\Lambda  =  \frac{1}{\sqrt{\mu}}H^{-1/2}(\Gamma)\cap \sqrt{{K}}H^{1/2}(\Gamma),\, 
&\mathbf{W} = \mathbf{V}\times\frac{1}{\sqrt{\mu}} L^2\left(\Omega_S\right)\times\sqrt{{K}} L^2\left(\Omega_D\right)\times \Lambda, \\
&\mbox{(EN)}\,\,
    &\Lambda  = \frac{1}{\sqrt{\mu}}H_{00}^{-1/2}(\Gamma)\cap \sqrt{{K}}H_{00}^{1/2}(\Gamma),\,
    &\mathbf{W} = \mathbf{V}\times\sqrt{{K}} L^2\left(\Omega_D\right) \times \Big(\frac{1}{\sqrt{\mu}} L^2\left(\Omega_S\right) \times \Lambda\Big)/\mathbb{R}, \\
&\mbox{(EN$^{*}$)}\,\,
    &\Lambda  = \frac{1}{\sqrt{\mu}}H_{00}^{-1/2}(\Gamma)\cap \sqrt{{K}}H_{00}^{1/2}(\Gamma),\,
    &\mathbf{W} = \mathbf{V}\times\sqrt{{K}} L^2\left(\Omega_D\right) \times \frac{1}{\sqrt{\mu}} L^2\left(\Omega_S\right) \times \Lambda.
  \end{align*}
  We recall that $\mathbf{V} = \mathbf{V}_S\times\mathbf{V}_D$ with the spaces $\mathbf{V}_S$, $\mathbf{V}_D$ are defined  in \eqref{Stokes-spaces} and \eqref{Darcy-spaces}, and the norm associated with the space $\mathbf{W}$ is
  \begin{align}
  \|\mathbf{u}, p, \lambda\|_{\mathbf{W}} 
  & = (\mu\|\mathbf{u}_S\|_{1,S}^2+\beta_{\tau}\|\mathbf{u}_S\cdot\bm{\tau}\|_{\Gamma}^2 + K^{-1}\|\mathbf{u}_D\|_{D}^2 + K^{-1}\|\nabla\cdot\mathbf{u}_D\|_{D}^2 \nonumber\\
  & \qquad + \mu^{-1}\|q_S\|_S^2 + K\|q_D\|_D^2 + \mu^{-1}\|\lambda\|_{-\frac12,\Gamma}^2 + K\|\lambda\|_{\frac12,\Gamma}^2)^{1/2}.
  \end{align}
\end{theorem}

We note that, due to the factor spaces in \Cref{remark-Stokes}, \Cref{remark-Darcy} and \Cref{Thm-wellposedness}, the problem data must satisfy compatibility conditions in order to ensure the solvability. For example, in \Cref{remark-Stokes}, we require that $(1, g_n)_{\Gamma}=0$.

\Cref{Thm-wellposedness} highlights a key challenge in ensuing uniform stability across different boundary condition configurations.  Specifically, in the NE and EN cases, the reliance on quotient spaces can complicate the practical implementation, as enforcing constraints directly in these spaces may require additional Lagrange multipliers which lead to dense blocks in the matrix of the discretized problem. To circumvent this difficulty, we next analyze the well-posedness of the Stokes-Darcy problem in the ``unconstrained'' space
\begin{align}\label{whole-space}
\widetilde{\mathbf{W}}= \sqrt{\mu}\,\mathbf{H}_{0, E}^1\left(\Omega_S\right) \cap \sqrt{\beta_{\tau}} \mathbf{L}_t^2(\Gamma) \times \frac{1}{\sqrt{{K}}} \mathbf{H}_{0, E}\left(\operatorname{div}; \Omega_D\right) \times \frac{1}{\sqrt{\mu}} L^2\left(\Omega_S\right) \times \sqrt{{K}} L^2\left(\Omega_D\right) \times \Lambda.
\end{align}
Here the use of standard $L^2$ simplifies implementation.
However, the space introduces a new challenge: inf-sup constant depends on physical parameters, which is discussed in the following theorem.

\begin{theorem}\label{thm-depend-para}
  For the NE and EN cases, the system is well-posed on the space $\widetilde{\mathbf{W}}$. Then the Stokes-Darcy operator $\mathcal{A}$ in \eqref{operator-A} is an isomorphism mapping $\widetilde{\mathbf{W}}$ to $\widetilde{\mathbf{W}}^{\prime}$ such that $\|\mathcal{A}\|_{\mathcal{L}\left(\widetilde{\mathbf{W}}, \widetilde{\mathbf{W}}^{\prime}\right)} \leq C_1$ and $\left\|\mathcal{A}^{-1}\right\|_{\mathcal{L}\left(\widetilde{\mathbf{W}}^{\prime}, \widetilde{\mathbf{W}}\right)} \leq$ $C_2^{-1}$, where $C_1$ and $C_2$ depend on the physical parameters $\mu$ and ${K}$.
\end{theorem}

\begin{proof}
  The continuity and coercivity conditions naturally hold, so we focus on the analysis of the inf-sup condition on the space $\widetilde{\mathbf{W}}$.  It is evident that the difference between $\mathbf{W}$ and $\widetilde{\mathbf{W}}$ lies in a one-dimensional space.
  Therefore, our focus is on establishing the 
  inf-sup condition within this space. We begin by presenting 
  the specific form of the inf-sup condition for each case.

In the NE case, the one-dimensional near-kernel space is characterized by $q_D = \mathbf{1}$ and $\phi = \mathbf{1}$. Let $\mathbf{u} = (\mathbf{u}_S, \mathbf{u}_D)$, then there is a constant $\gamma = C\min\big\{\frac{1}{\sqrt{\mu K}},1 \big\}$ such that
    \begin{align}\label{inf-sup-ND}
        & \inf_{q_S\in Q_S}\sup_{0\neq\mathbf{u}\in\mathbf{V}}\frac{(B_S\mathbf{u}_S,q_S) +(B_D\mathbf{u}_D,\mathbf{1}) + (T_S\mathbf{u}_S,\mathbf{1}) + (T_D\mathbf{u}_D,\mathbf{1})}{(\|\mathbf{u}_S\|_{\mathbf{V}_S}^2+ \|\mathbf{u}_D\|_{\mathbf{V}_D}^2)^{1/2} (\|q_S\|_{Q_S}^2 + \|\mathbf{1}\|_{Q_D}^2 + \|\mathbf{1}\|_{\Lambda}^2)^{1/2}} \geq \gamma. 
    \end{align}
 Similarly, in the EN case, the one-dimension near-kernel space is $q_S = \mathbf{1}, \phi = \mathbf{1}$. The inf-sup condition ensures the existence of a constant $\gamma = C\min\big\{\sqrt{\mu K},1 \big\}$ such that
    \begin{align}\label{inf-sup-DN}
        & \inf_{q_D\in Q_D}\sup_{0\neq\mathbf{u}\in\mathbf{V}}\frac{(B_S\mathbf{u}_S,\mathbf{1}) +(B_D\mathbf{u}_D,q_D) + (T_S\mathbf{u}_S,\mathbf{1})+ (T_D\mathbf{u}_D,\mathbf{1})}{(\|\mathbf{u}_S\|_{\mathbf{V}_S}^2+ \|\mathbf{u}_D\|_{\mathbf{V}_D}^2)^{1/2} (\|\mathbf{1}\|_{Q_S}^2 + \|q_D\|_{Q_D}^2 + \|\mathbf{1}\|_{\Lambda}^2)^{1/2}} \geq \gamma.
    \end{align}

In the following, we first establish the inf-sup condition \eqref{inf-sup-ND} for the NE case in detail. To begin, we construct an auxiliary Stokes subproblem: given $q_S$ and $\phi$, find $\mathbf{u}_S$ such that
    \begin{subequations}\label{cons-us}
    \begin{alignat}{2}
         2\mu(\bm{\epsilon}(\mathbf{u}_S), \bm{\epsilon}(\mathbf{v}_S))_S + \beta_{\tau}(T_t \mathbf{u}_S, T_t \mathbf{v}_S)_{\Gamma} - (\nabla\cdot\mathbf{v}_S,p_S)_S + (T_S\mathbf{v}_S,\lambda_S)_{\Gamma} = 0, &\quad \forall\,\mathbf{v}_S\in\mathbf{V}_S, \label{cons-us-1}\\
         -(\nabla\cdot\mathbf{u}_S,r_S)_S = \mu^{-1}(q_S, r_S)_S,& \quad  \forall\,r_S\in Q_S, \label{cons-us-2}\\
         (T_S\mathbf{u}_S,\mu_S)_{\Gamma} = \mu^{-1}(\phi,\mu_S)_{\Gamma},  &\quad \forall\,\mu_S\in\Lambda.\label{cons-us-3}
    \end{alignat}
    \end{subequations}
    According to the Stokes inf-sup condition, we have the following estimate 
    \begin{align}\label{bound-us}
        \|\mathbf{u}_S\|_{\mathbf{V_S}} \lesssim  \mu^{-1}\left(\|q_S\|_{S}^2 + \|\phi\|_{-\frac{1}{2},\Gamma}^2\right)^{1/2}.
    \end{align}
    Using the fact $q_D = \mathbf{1}$ and $\phi = \mathbf{1}$, and applying integration by parts, we obtain
    \begin{align}\label{cons-ud}
        (B_D\mathbf{u}_D, q_D) + (T_D\mathbf{u}_D, \phi) = -(\nabla\cdot\mathbf{u}_D, \mathbf{1})_D + (\mathbf{u}_D\cdot\mathbf{n}_D, \mathbf{1})_{\Gamma} = (\mathbf{u}_D, \nabla\mathbf{1})_D = 0.
    \end{align}
    According to \eqref{cons-us-2}-\eqref{cons-us-3}, \eqref{cons-ud}, and the norm equivalence, we derive the following inequality,
    \begin{align*}
        & (B_S\mathbf{u}_S,q_S) +(B_D\mathbf{u}_D,\mathbf{1}) + (T_S\mathbf{u}_S,\mathbf{1})+ (T_D\mathbf{u}_D,\mathbf{1}) \nonumber \\
        &  = - (\nabla\cdot\mathbf{u}_S, q_S)_D + (\mathbf{u}_S\cdot\mathbf{n}_S, \mathbf{1})_{\Gamma}  = \mu^{-1}\|q_S\|_{S}^2 + \mu^{-1} \|\mathbf{1}\|_{-\frac{1}{2},\Gamma}^2 \nonumber \\
        & \gtrsim \mu^{-1}\|q_S\|_{S}^2 + \frac{1}{3}\mu^{-1} \|\mathbf{1}\|_{-\frac{1}{2},\Gamma}^2 +  \frac{1}{3}(\mu {K})^{-1} {K}\|\mathbf{1}\|_{D}^2
         +  \frac{1}{3}(\mu {K})^{-1} {K} \|\mathbf{1}\|_{\frac{1}{2},\Gamma}^2 \nonumber\\
         &  \gtrsim \min\left\{(\mu {K})^{-1},1\right\} \left(\mu^{-1}\|q_S\|_{S}^2 + \mu^{-1} \|\mathbf{1}\|_{-\frac{1}{2},\Gamma}^2 +   {K}\|\mathbf{1}\|_{D}^2
         + {K} \|\mathbf{1}\|_{\frac{1}{2},\Gamma}^2\right).
    \end{align*} 
    Then letting $\mathbf{u}_D = \mathbf{0}$, combining this with \eqref{bound-us}, we obtain the inf-sup condition \eqref{inf-sup-ND} by choosing $\gamma^2 = C\min\left\{(\mu K)^{-1},1\right\}$.

    For the EN case, the proof is similar. We construct the Darcy subproblem as follows,
    \begin{subequations}\label{cons-ud-DN}
    \begin{alignat}{2}
         K^{-1}(\mathbf{u}_D,\mathbf{v}_D)_D - (\nabla\cdot\mathbf{v}_D,p_D)_D + (T_D\mathbf{v}_D,\lambda_D)_{\Gamma} = 0, &\quad \forall\,\mathbf{v}_D\in\mathbf{V}_D, \label{cons-ud-1}\\
         -(\nabla\cdot\mathbf{u}_D,r_D)_D= K(q_D, r_D)_D, &\quad  \forall\, r_D\in Q_D, \label{cons-ud-2}\\
         (T_D\mathbf{u}_D,\mu_D)_{\Gamma} = K(\phi,\mu_D)_{\Gamma},  &\quad \forall\,\mu_D\in\Lambda.\label{cons-ud-3}
    \end{alignat}
    \end{subequations}
    According the Darcy inf-sup condition, we have
    \begin{align*}
        \|\mathbf{u}_D\|_{\mathbf{V_D}} \lesssim  K\left(\|q_D\|_{D}^2 + \|\phi\|_{\frac{1}{2},\Gamma}^2\right)^{1/2}.
    \end{align*}
    In this case, the one-dimensional near-kernel space is characterized by $q_S = \mathbf{1}$ and $\phi = \mathbf{1}$ and, thus, we have the following estimate 
    \begin{align*}
        & (B_S\mathbf{u}_S,\mathbf{1}) +(B_D\mathbf{u}_D,q_D) + (T_S\mathbf{u}_S,\mathbf{1})+ (T_D\mathbf{u}_D,\mathbf{1}) \nonumber \\
        &  = - (\nabla\cdot\mathbf{u}_D, q_D)_D + (\mathbf{u}_D\cdot\mathbf{n}_D, \mathbf{1})_{\Gamma} = {K}\|q_D\|_{D}^2 + {K}\|\mathbf{1}\|_{\frac{1}{2},\Gamma}^2 \nonumber \\
        & \gtrsim {K}\|q_D\|_{D}^2 + \frac{1}{3} (\mu {K}) \mu^{-1} \|\mathbf{1}\|_{S}^2 +  \frac{1}{3}(\mu {K})\mu^{-1}\|\mathbf{1}\|_{-\frac{1}{2},\Gamma}^2
         +  \frac{1}{3}{K}\|\mathbf{1}\|_{\frac{1}{2},\Gamma}^2 \nonumber\\
         & \gtrsim  \min\left\{ \mu {K}, 1\right\}\left(\mu^{-1}\|\mathbf{1}\|_{S}^2 + \mu^{-1} \|\mathbf{1}\|_{-\frac{1}{2},\Gamma}^2 +   {K}\|q_D\|_{D}^2
         + {K} \|\mathbf{1}\|_{\frac{1}{2},\Gamma}^2\right).
    \end{align*}
    We conclude the proof of the inf-sup condition \eqref{inf-sup-DN} by setting 
    $\mathbf{u}_S = 0$. 
\end{proof}

To summarize, we have obtained the well-posedness results for all cases of the coupled problem
\cref{weak-form}. For the NN, EE, NE$^{*}$, and EN$^*$ cases, the inf-sup constants are independent of both physical and
discretization parameters. However, for the NE and EN cases, the inf-sup constants depend on the physical parameters 
in the ``unconstrained'' space \cref{whole-space}, whereas they become parameter-independent in the space that excludes the one-dimensional near-kernel space.

\subsection{Preconditioning}
Building on the well-posedness results from \cref{Thm-wellposedness} and \cref{thm-depend-para}, and leveraging the operator preconditioning framework \cite{Kent2011Preconditioning}, the canonical preconditioner for the coupled Stokes-Darcy problem is given by
\begin{equation}\label{preconditioner}
\mathcal{B} = \operatorname{diag}(-2\mu\nabla\cdot\bm{\epsilon} + \beta_{\tau} T_t^{\prime} T_t,\, {K}^{-1} (I -\nabla\nabla\cdot),\, (2\mu)^{-1}I,\,  KI,\, S)^{-1},
\end{equation}
where $S$ depends on the boundary condition that intersects the interface and is defined as follows:
\begin{equation}\label{eq:interface_preconditioner}
  \begin{aligned}
    &(\text{NN}) \,\,
    &S = \mu^{-1}(-\Delta + I)^{-1/2} + {K} (-\Delta + I)_{00}^{1/2},\\
    &(\text{EE}) \,\, 
    &S = \mu^{-1}(-\Delta + I)_{00}^{-1/2} + {K} (-\Delta + I)^{1/2},\\
    &(\text{NE}) \ \text{and} \ (\text{NE}^{*}) \,\,
    &S = \mu^{-1}(-\Delta + I)^{-1/2} + {K} (-\Delta + I)^{1/2},\\
    &(\text{EN}) \ \text{and} \ (\text{EN}^{*}) \,\, 
    &S = \mu^{-1}(-\Delta + I)_{00}^{-1/2} + {K} (-\Delta + I)_{00}^{1/2}.
  \end{aligned}
\end{equation}
\Cref{Thm-wellposedness} leads to the following estimates about the
conditioning of the Stokes-Darcy problem in NN and EE cases with preconditioner
\eqref{preconditioner}.

\begin{theorem}\label{thm-cond-num}
   For NN, EE, NE\,$^{*}$ and EN\,$^*$ cases, the condition number of the preconditioned operator satisfies $\kappa(\mathcal{BA}) = \mathcal{O}(1)$.
\end{theorem}

From the results of \cref{thm-depend-para}, it is evident that the smallest eigenvalue in magnitude—and consequently, the condition number—scales with the physical parameters in the NE and EN cases on the ``unconstrained” space $\widetilde{\mathbf{W}}$ \cref{whole-space}.  However, the effective condition number $\kappa_{\text{eff}}$ is independent of the physical parameters when the smallest eigenvalue in magnitude is removed, as in the space $\mathbf{W}$ defined in \cref{Thm-wellposedness}. 
\begin{theorem}\label{thm-eff-cond-num}
    For the NE and EN cases, the condition number of the preconditioned operator satisfies $\kappa(\mathcal{BA}) := \|\mathcal{B}\mathcal{A}\|_{\widetilde{\mathbf{W}}\mapsto\widetilde{\mathbf{W}}} \|(\mathcal{B}\mathcal{A})^{-1}\|_{\widetilde{\mathbf{W}}\mapsto\widetilde{\mathbf{W}}} = \mathcal{O}(\min\{ 1/ \sqrt{\mu K}, 1 \})$ or $\mathcal{O}(\min\{ \sqrt{\mu K}, 1 \})$. However, the effective condition number satisfies $\kappa_{\text{eff}}(\mathcal{BA}) :=  \|\mathcal{B}\mathcal{A}\|_{\mathbf{W}\mapsto\mathbf{W}} \|(\mathcal{B}\mathcal{A})^{-1}\|_{\mathbf{W}\mapsto\mathbf{W}} = \mathcal{O}(1).$
\end{theorem}

\begin{remark}
  By the Courant–Fischer–Weyl min-max principle, \cref{Thm-wellposedness} and \cref{thm-eff-cond-num} indicate that we can divide $\sigma(\mathcal{BA})$, the spectrum of $\mathcal{BA}$, into two parts: $\{\lambda_0\}$ and $\sigma_1(\mathcal{BA})$.  The eigenvalue $\lambda_0$, which has the smallest absolute value, depends on the physical parameters $\mu$ and ${K}$. On the other hand, $\sigma_1(\mathcal{BA}) \subset [-C, -\gamma] \cup [\gamma, C]$ contains the remaining eigenvalues that are bounded with respect to the parameters.  Based on the theory of the MINRES method~\cite{simoncini2013superlinear}, the convergence rate  can be split into two parts; 1) a period of stagnation related to the removal of the eigencomponents associated with the smallest eigenvalue and 2) the convergence rate predicted  by $(\kappa_{\mathrm{eff}}(\mathcal{BA})-1)/(\kappa_{\mathrm{eff}}(\mathcal{BA})+1)$, which is parameter-robust. 
\end{remark}

\subsection{Robustness of discrete preconditioner}\label{sec3-2:preconditioner}
To test numerically the results of \Cref{thm-cond-num} and \Cref{thm-eff-cond-num}, we consider triangulations $\mathcal{T}_h$ of $\Omega$ (parameterized by mesh size $h$), which is conforming with the interface $\Gamma$ in the sense that for each element $K\in\mathcal{T}_h$, the intersection $\overline{K}\cap\Gamma$ is either a vertex or a facet of $K$. Let $\mathcal{T}_{S,h}=\mathcal{T}_h\cap\Omega_S$, the Stokes variables are discretized by the lowest order Taylor-Hood pair, i.e.,
\begin{equation}\label{eq:stokes_h}
\begin{aligned}
  \mathbf{V}_{S, h} &= \left\{\mathbf{v}\in\mathbf{H}^1(\Omega_S)\,|\, \mathbf{v}|_{K}\in[\mathbb{P}_2(K)]^d\ \ \forall\,K\in\mathcal{T}_{S, h}\right\},\\
  Q_{S, h} &= \left\{q \in {H}^1(\Omega_S)\,|\, {q}|_{K}\in \mathbb{P}_1(K) \ \ \forall\, K\in\mathcal{T}_{S, h}\right\},
\end{aligned}
\end{equation}
where $\mathbb{P}_k(K)$ denotes the space of polynomials of degree up to $k$ on element $K$.
The variables in the Darcy domain are discretized using the lowest-order Raviart-Thomas element (denoted by $\mathbb{RT}_0$) and the piecewise constant element, i.e.,
\begin{equation}\label{eq:darcy_h}
\begin{aligned}
  \mathbf{V}_{D, h} &= \left\{\mathbf{v}\in\mathbf{H}(\text{div}, \Omega_D)\,|\, \mathbf{v}|_{K}\in \mathbb{RT}_0(K)\ \ \forall\,K\in\mathcal{T}_{D, h}\right\},\\
  Q_{D, h} &= \left\{q \in {L}^2(\Omega_D)\,|\, {q}|_{K}\in \mathbb{P}_0(K) \ \ \forall\, K\in\mathcal{T}_{D, h}\right\}.
\end{aligned}
\end{equation}
 Letting $\mathcal{F}_h$
be the collection of all facets of elements in $\mathcal{T}_h$ and $\Gamma_h=\mathcal{F}_h\cap\Gamma$,
the discrete Lagrange multiplier space $\Lambda_h$ is constructed as
\begin{equation*}
\begin{aligned}
  \Lambda_h &= \left\{ \lambda \in L^2(\Gamma)\,|\, \lambda|_{F}\in \mathbb{P}_0(F)\ \ \forall\,F\in\Gamma_{h}\right\}.
\end{aligned}
\end{equation*}
Finally, we set
$\mathbf{W}_h=\mathbf{V}_{S, h}\times\mathbf{V}_{D, h}\times Q_{S, h}\times Q_{D, h}\times \Lambda_h$.

We implemented the above discretization using the FEniCS finite element library \cite{logg2012automated} and its extension FEniCS\textsubscript{ii} \cite{kuchta2020assembly} to handle the coupling at the interface. Since the spaces $\mathbf{V}_{S, h}$ and $\mathbf{V}_{D, h}$ are $H^1$- and $H(\text{div})$-conforming, respectively,  the discrete operators $\mathcal{A}_h$ and $\mathcal{B}_h$ are obtained by applying $\mathcal{A}$ and $\mathcal{B}$ to $\mathbf{W}_h\subset \mathbf{W}$, respectively. We note that the discrete representation of  the multiplier preconditioner $S$ in \eqref{eq:interface_preconditioner} is computed using a spectral/Fourier representation, see, for example, \cite{boon2022robust}.

Before conducting further numerical experiments, we verify correctness of our implementation by examining the approximation properties of the discretization.
\begin{example}[Error convergence analysis]\label{ex:mms}
We set the Stokes domain to be $\Omega_S = (0, 1) \times (0, 1)$ and the Darcy domain $\Omega_D = (0, 1)\times
(1, 2)$, such that the interface $\Gamma = (0, 1)\times\{1\}$. As the solution of \eqref{Stokes-model}, \eqref{Darcy-model}, \eqref{interface} we
consider 
\begin{align}
& \mathbf{u}_S = \operatorname{curl}\psi, \ \psi=\cos(\pi(x+y)),
\qquad p_S = \sin(2\pi(x-y)), \quad \text{ in } \Omega_S; \label{eq:fenics_mms} \\
& p_D = \sin(2\pi(x-2y)), \quad\text{ in }\Omega_D. \label{eq:fenics_mmd}
\end{align}
The Darcy flux $\mathbf{u}_D$ is computed from its definition based on the pressure $p_D$. We note that the solution \eqref{eq:fenics_mms}-\eqref{eq:fenics_mmd} requires additional source terms on the right-hand side of the coupling conditions. This data, alone with the volume source terms $\bm{f}_S$ and $g_D$ is set according to the material parameters given as $\mu=3$, ${K}=1$, $\alpha_{\text{BJS}}=0.5$. Finally, we prescribe mixed boundary conditions on the boundaries
\begin{align*}
&\Gamma^N_D=\left\{(x, y)\in\partial\Omega_D, y = 2\right\}, \quad\Gamma^E_D=\left\{(x, y)\in\partial\Omega_D, x(1-x) = 0\right\},\\
&\Gamma^E_S=\left\{(x, y)\in\partial\Omega_S, y = 0\right\}, \quad \Gamma^N_S=\left\{(x, y)\in\partial\Omega_S, x(1-x) = 0\right\}.
\end{align*}

Convergence results of $\mathbb{P}_2$-$\mathbb{P}_1$ and $\mathbb{RT}_0$-$\mathbb{P}_0$ based approximations of the Stokes and Darcy variables, obtained on a sequence of uniformly refined meshes,  are presented in \Cref{tab:mms_cvrg}. We observe quadratic
convergence for the variables in the fluid domain and linear convergence for the
variables in the porous domain. These convergence rates are theoretically optimal.

\begin{table}
\setlength{\tabcolsep}{4pt}  
  \centering
  \footnotesize{
    \begin{tabular}{c|cc|cc}
      \hline
      $h$ &
      $\lVert \nabla (\mathbf{u}_{S, h}-\mathbf{u}_S)\rVert_{0, \Omega_S}$ &
      $\lVert p_{S, h}-p_S\rVert_{0, \Omega_S}$ & 
      $\lVert \nabla\cdot (\mathbf{u}_{D, h}-\mathbf{u}_D)\rVert_{0, \Omega_D}$ &
      $\lVert p_{D, h}-p_D\rVert_{0, \Omega_D}$ \\
      \hline
3.54E-01 & 1.1888E+00(--)   & 2.2983E-01(--)   & 2.8014E+02(--)   & 1.1339E+00(--)\\  
1.77E-01 & 2.9179E-01(2.03) & 4.1099E-02(2.48) & 1.4885E+02(0.91) & 3.6369E-01(1.64)\\ 
8.84E-02 & 7.3164E-02(2.00) & 8.7034E-03(2.24) & 7.5271E+01(0.98) & 1.7232E-01(1.08)\\
4.42E-02 & 1.8330E-02(2.00) & 2.0667E-03(2.07) & 3.7730E+01(1.00) & 8.5300E-02(1.01)\\ 
2.21E-02 & 4.5881E-03(2.00) & 5.0960E-04(2.02) & 1.8877E+01(1.00) & 4.2546E-02(1.00)\\ 
1.10E-02 & 1.1478E-03(2.00) & 1.2695E-04(2.01) & 9.4398E+00(1.00) & 2.1260E-02(1.00)\\
\hline
\end{tabular}}
\caption{Error convergence for the discretization \eqref{eq:stokes_h}, \eqref{eq:darcy_h}
  of the Stokes-Darcy problem setup in \Cref{ex:mms}. Estimated convergence rate
  is shown in the brackets.
}
\label{tab:mms_cvrg}
\vskip -20pt
\end{table}
\end{example}

To assess robustness of the preconditioner,  we typically consider the spectral condition number of the discrete preconditioned problem. The associated generalized eigenvalue problem $\mathcal{A}_hx=\lambda \mathcal{B}^{-1}_h x$ is solved iteratively using SLEPc \cite{Hernandez:2005:SSF} for the largest eigenvalue in magnitude and the smallest (for $\kappa$) or the second smallest (for $\kappa_{\text{eff}}$) eigenvalues in magnitude. Additionally, we will discuss  performance of the preconditioners in terms of boundedness of the MINRES iterations. In this context, we apply the exact preconditioner with the diagonal blocks inverted using LU decomposition. However, we remark that all the blocks can be approximated by scalable optimal methods. Specifically, the fractional interface preconditioners can be efficiently implemented, for example, through rational approximations \cite{budivsa2022rational, adler2024improving}.

\begin{example}[Configurations without a near-kernel]\label{ex:noker}
 Using the domain setup from \Cref{ex:mms} and the discretization given by \eqref{eq:stokes_h} and \eqref{eq:darcy_h} we consider the boundary condition configurations (a)-(d) from \Cref{fig:bc_config}. Following \Cref{thm-cond-num}, we expect that the preconditioner \eqref{preconditioner}, with the interface operator defined by \eqref{eq:interface_preconditioner}, will results in condition numbers that are bounded with respect to the mesh size and the parameters when boundary conditions EE and NN are applied. This behavior is indeed observed in \Cref{fig:cond_noker}. Furthermore, the results are bounded also for the EN$^*$ and NE$^*$ cases, where mixed boundary conditions set on $\partial\Omega_S\setminus\Gamma$ and $\partial\Omega_D\setminus\Gamma$ prevent the near-kernel behavior seen in the EN and NE cases.
  
\begin{figure}[h!]
  \setlength{\abovecaptionskip}{0.0625\baselineskip}    
  \setlength{\belowcaptionskip}{0.0625\baselineskip}
  \centering
  \begin{subfigure}[c]{\textwidth}
    \includegraphics[width=\textwidth]{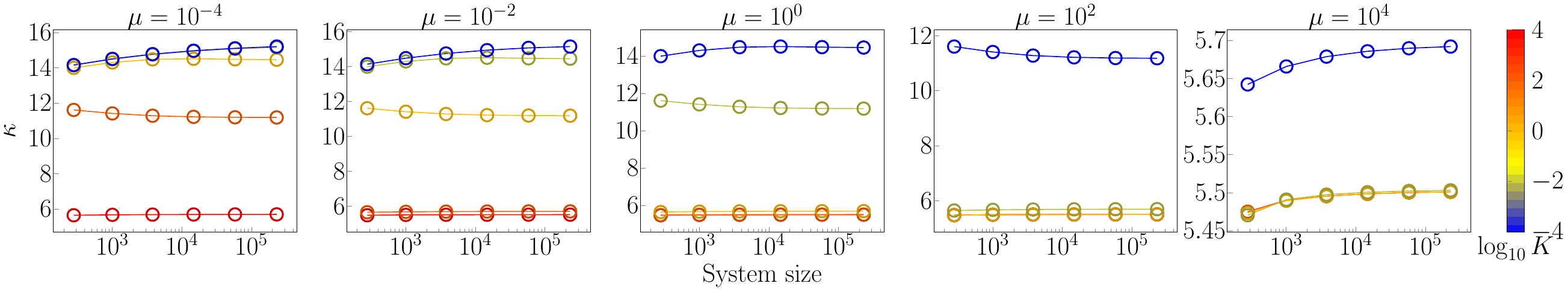}
   \caption{NN}\label{fig:cond_NN}
  \end{subfigure}
    \begin{subfigure}[c]{\textwidth}
      \includegraphics[width=\textwidth]{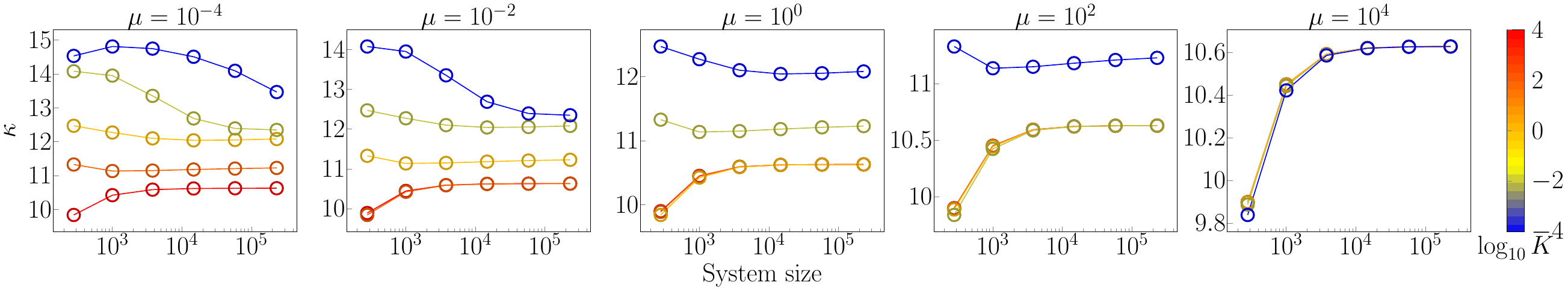}
    \caption{EE}\label{fig:cond_EE}
    \end{subfigure}
    \begin{subfigure}[c]{\textwidth}
      \includegraphics[width=\textwidth]{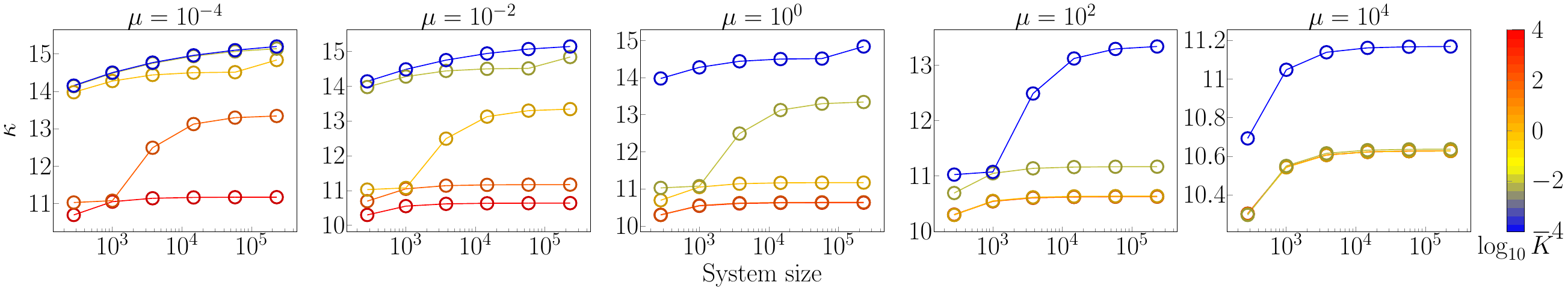}
    \caption{NE$^*$}\label{fig:cond_NEnoker}
    \end{subfigure}
    \begin{subfigure}[c]{\textwidth}
      \includegraphics[width=\textwidth]{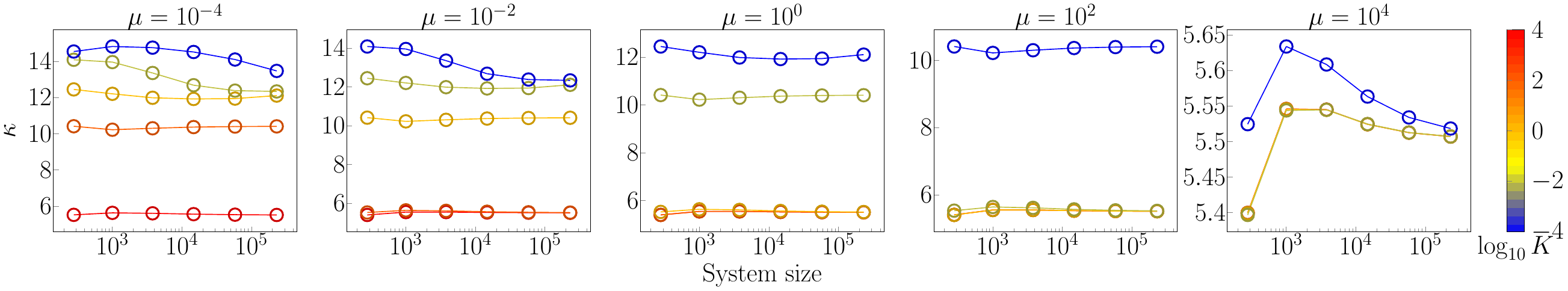}
    \caption{EN$^*$}\label{fig:cond_ENnoker}
    \end{subfigure}
    \caption{Performance of preconditioner \eqref{preconditioner} for the
      Stokes-Darcy problem, applied to the geometry from \Cref{ex:mms}, under different boundary condition configurations (see also \Cref{fig:bc_config}) and parameter variations. We set $\alpha_{\text{BJS}}=0.5$. Defining the multiplier preconditioner $S$ as in \eqref{eq:interface_preconditioner} yields to robustness in $\mu$ and $K$ as the considered boundary condition configurations do not lead to near-kernel.
    }
  \label{fig:cond_noker}
\end{figure}
\end{example}

We next investigate the EN and NE cases which have the near-kernel.

\begin{example}[Near-kernel configurations]\label{ex:ker}
We consider the domain and discretization setup from \Cref{ex:noker} with the EN and NE boundary conditions configurations. In agreement with \Cref{thm-cond-num}, we observe in \Cref{fig:cond_ker} that the condition numbers with preconditioner \eqref{preconditioner} deteriorate based on $\mu K$. Following \Cref{thm-eff-cond-num}, we evaluate the preconditioner in terms of the effective condition number. In \Cref{fig:cond2_ker}, it can be seen that removing the subspace spanned by the eigenmode associated with the $\mu K$-dependent smallest eigenvalue in magnitude leads to a condition number that remains bounded with respect to variations in $\mu$ and $K$.

\begin{figure}[h!]
  \setlength{\abovecaptionskip}{0.0625\baselineskip}    
  \setlength{\belowcaptionskip}{0.0625\baselineskip}
  \centering
    \begin{subfigure}[c]{\textwidth}
      \includegraphics[width=\textwidth]{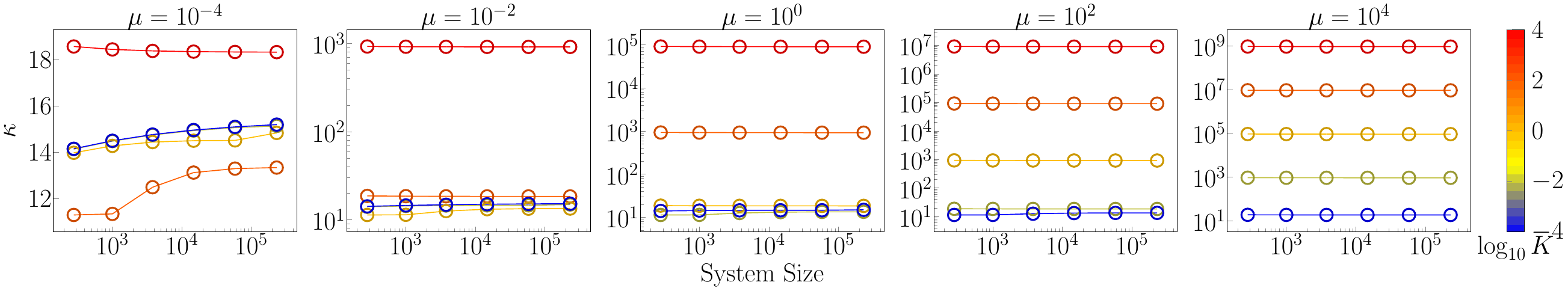}
    \caption{NE}\label{fig:NE_cond_ker}
    \end{subfigure}
    \begin{subfigure}[c]{\textwidth}
      \includegraphics[width=\textwidth]{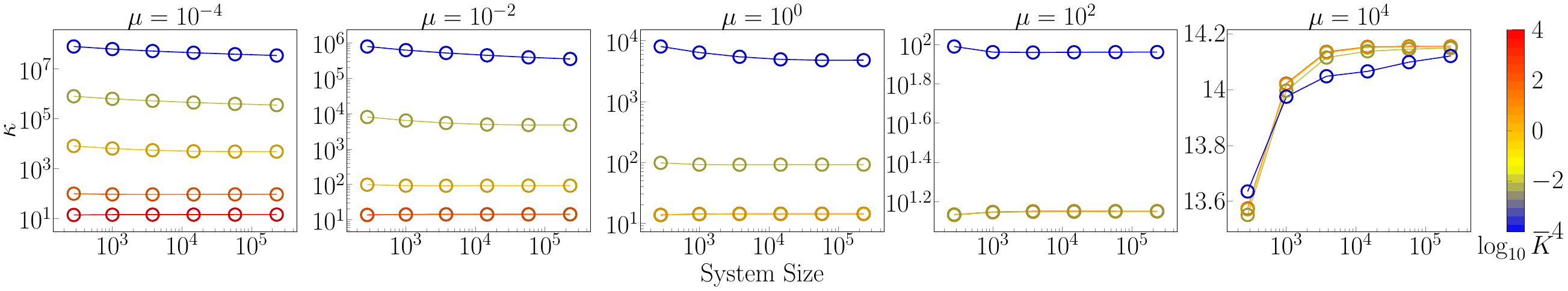}
    \caption{EN}\label{fig:EN_cond_ker}
    \end{subfigure}
    \caption{Performance of preconditioner \eqref{preconditioner} for the
      Stokes-Darcy problem, applied to the geometry from \Cref{ex:mms}, with different parameter
      values for the EN and NE cases (see also \Cref{fig:bc_config}).
      Condition numbers are sensitive to the product $\mu K$ due to the present near-kernel.
    }
  \label{fig:cond_ker}
\end{figure}

\begin{figure}[h!]
  \setlength{\abovecaptionskip}{0.0625\baselineskip}    
  \setlength{\belowcaptionskip}{0.0625\baselineskip}
  \centering
    \begin{subfigure}[c]{\textwidth}
      \includegraphics[width=\textwidth]{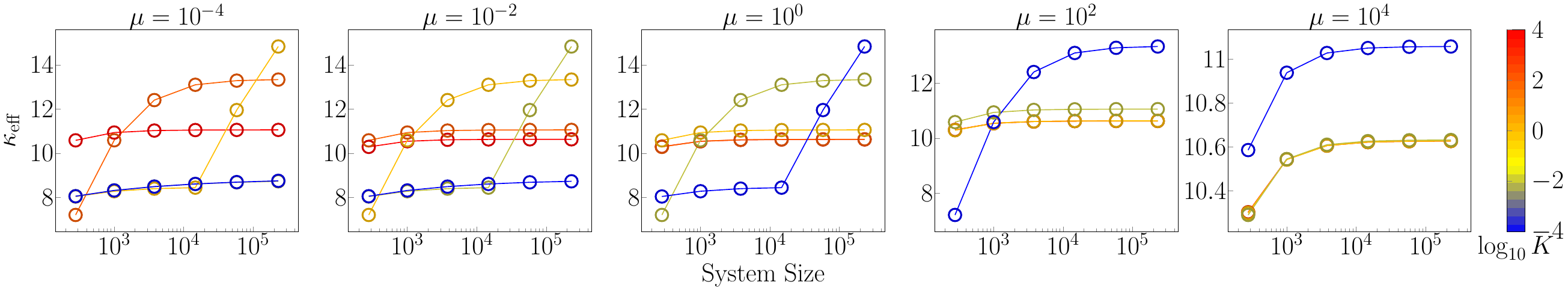}
    \caption{NE}\label{fig:NE_cond2_ker}
    \end{subfigure}
    \begin{subfigure}[c]{\textwidth}
      \includegraphics[width=\textwidth]{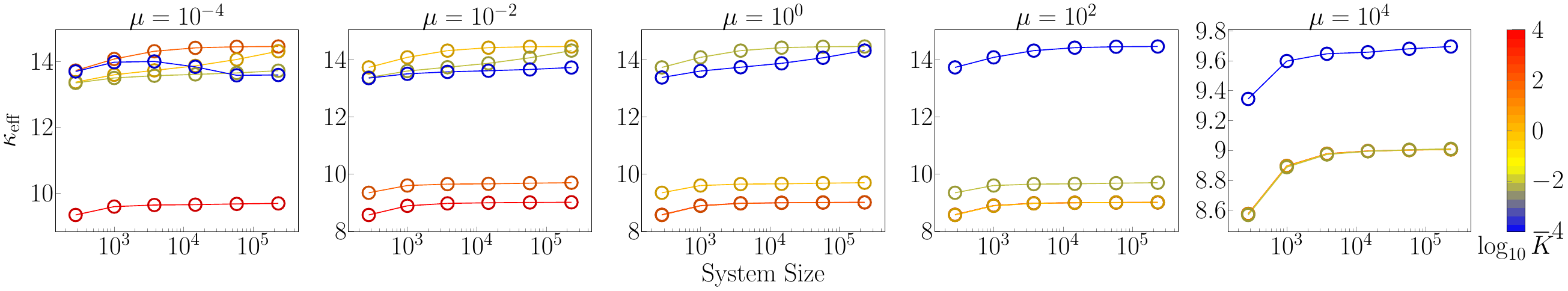}
    \caption{EN}\label{fig:EN_cond2_ker}
    \end{subfigure}
    \caption{
      Performance of the Stokes-Darcy preconditioner in \Cref{fig:cond_ker} in terms
      of the effective condition number based on the second smallest-in-magnitude eigenvalue.
      Comparisong with the rightmost panels for the NE case and leftmost panels for the EN case in
      \Cref{fig:cond_ker} reveals the parameter-sensitivity of the smallest eigenvalue.
    }
  \label{fig:cond2_ker}
\end{figure}
\end{example}

In the final example, we illustrate the impact of the single degenerate eigenvalue of the \eqref{preconditioner}-preconditioned system on the MINRES convergence.  

\begin{example}[MINRES convergence]\label{ex:nodef_cvrg}
 We consider the manufactured test problem from \Cref{ex:mms} with EN configuration of boundary conditions. Setting $\alpha_{\text{BJS}}=0.5$ and $K=10^{-4}$, we compare the convergence of the \eqref{preconditioner}-preconditioned MINRES solver when $\mu=10^{4}$ or $\mu=10^{-4}$. In both cases, the solver starts from zero initial guess and converges once the preconditioned residual norm is reduced by factor $10^{12}$. From \Cref{fig:cond_ker} and \Cref{fig:cond2_ker}, we observe that two cases have similar condition numbers, namely, $\kappa\approx 14.1$ for $\mu=10^4$ and $\kappa_{\text{eff}}\approx 13.6$ for $\mu=10^{-4}$. However, \Cref{fig:cvrg_EN} shows very different characteristics of MINRES convergence between the two cases. In particular, $\mu=10^{-4}$ exhibits prolonged stagnation/plateau regions in the convergence curves.

\begin{figure}
  \centering
  \includegraphics[width=0.8\textwidth]{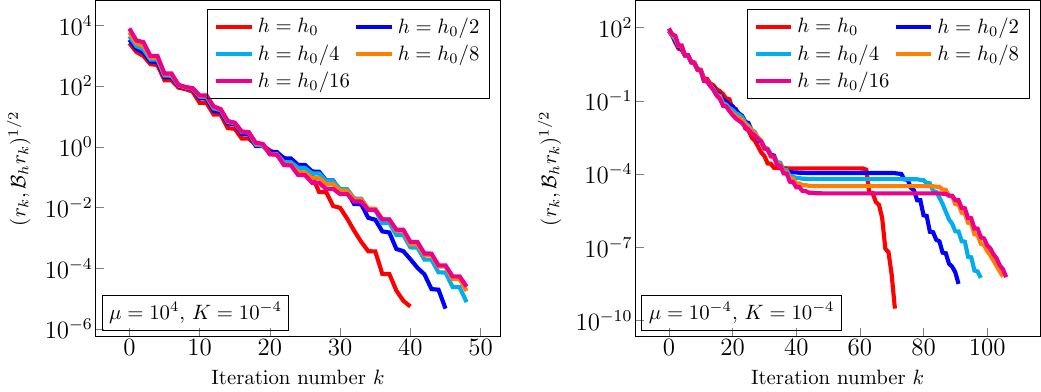}
  \caption{
    Evolution of preconditioned residual norm over MINRES iterations for the Stokes-Darcy problem from \Cref{ex:mms} for the EN case. Following \Cref{thm-cond-num} and \Cref{thm-eff-cond-num}, parameter values $K=10^{-4}$ and $\mu=10^{-4}$ yield a large condition number, while the effective condition number is bounded. However, convergence of the MINRES solver is considerably delayed compared to the setup with $K=10^{-4}$ and$\mu=10^{4}$.
  }
  \label{fig:cvrg_EN}
\end{figure}
\end{example}

Motivated by \Cref{ex:nodef_cvrg}, we will next discuss the convergence process of MINRES method for the two near-kernel cases.

\section{MINRES convergence and deflation preconditioner}\label{sec4:discrete analysis}
In this section, we study the behavior of the MINRES method when there is a slow-to-converge eigenvalue of the linear system.
\subsection{Convergence analysis}
Consider a large and sparse linear system of the form 
\begin{align}\label{iter-system}
    \mathcal{A}x = b,
\end{align}
where the $n\times n$ matrix $\mathcal{A}$ is symmetric and indefinite. Given the initial guess $x_0$ to the solution of \eqref{iter-system}, and the corresponding initial residual $r_0 = b - \mathcal{A}x_0$, the Krylov subspace of dimension $m$ defined by $\mathcal{A}$ and $r_0$ is given by 
 $  \mathcal{K}_m(\mathcal{A},r_0) = \operatorname{span}\{r_0, \mathcal{A}r_0,..., \mathcal{A}^{m-1}r_0\}$.
At the $m$-th step, an approximation to the solution of \eqref{iter-system}, $x_m$, can be obtained in $x_0 + \mathcal{K}_m(\mathcal{A},r_0)$, by imposing some additional conditions. The residual $r_m$ is related to the initial residual as $r_m = p_m(\mathcal{A}) r_0,$ where the residual polynomial $p_m\in \mathbb{P}_m$, the set of polynomials of degree at most $m$ such that $p_m(0) = 1$. Let $\Lambda(\mathcal{A}) = \{\lambda_1, \lambda_2,...,\lambda_n\}$ be the set of eigenvalues and $\{v_1, v_2,...,v_n\}$ be the corresponding orthonormal eigenvectors of $\mathcal{A}$, with the eigenvalues ordered increasingly in magnitude. Then we have the following standard bound \cite{simoncini2013superlinear} 
\begin{align*}
    \| r_m \| = \|p_m(\mathcal{A}) r_0\| \leq \min_{p\in \mathbb{P}_m} \max_{1\leq i\leq n} |p(\lambda_i)| \|r_0\|.
\end{align*}
A more specific bound was developed in \cite[Section 3.1]{Greenbaum1997Iterative} for the case where $\Lambda(\mathcal{A})\subset[a,b]\cup[c,d]$, $a<b<0<c<d$. Under the constraint that $|b-a| = |d-c|$, using an appropriate transformation of the intervals and bounds on Chebychev polynomials, the following estimate holds
\begin{align}\label{bound-rm}
    \frac{\|r_m\|}{\|r_0\|} \leq 2\left(\frac{\sqrt{|ad|}- \sqrt{|bc|}}{\sqrt{|ad|} + \sqrt{|bc|}}\right)^{[m/2]},
\end{align}
where $[m/2]$ is the integer part of $m/2$. 

Following \cite{simoncini2013superlinear}, let us now assume that $\lambda_1$ is close to zero, and the remaining eigenvalues belong to $[a,b]\cup[c,d]$. Let $r_0 = \sum\limits_{k = 1}^n \alpha_k v_k$, where $\alpha_k = (r_0, v_k), k =1,...,n$. Then, we have
\begin{align*}
    r_m & = p_m(\mathcal{A})r_0 = \sum_{k=1}^n p_m(\lambda_k)\alpha_k v_k = p_m(\lambda_1)\alpha_1 v_1 + \sum_{k=2}^n p_m(\lambda_k)\alpha_k v_k \overset{\Delta}{=} p_m(\lambda_1)\alpha_1 v_1 + \Bar{r}_0. 
\end{align*}
Let $\Bar{q}_j$ be the MINRES polynomial in $\mathcal{K}_{j}(\mathcal{A},\Bar{r}_0)$, then
$ \Bar{r}_j = \Bar{q}_j(\mathcal{A}) \Bar{r}_0 = \sum_{k=2}^n \Bar{q}_j(\lambda_k)p_m(\lambda_k)\alpha_k v_k$.
Let $\theta_1^{(m)}$ be the harmonic Ritz value (see  \cite{simoncini2013superlinear,morgan1991computing,paige1995approximate}) closet to $\lambda_1$ in $\mathcal{K}_m(\mathcal{A},r_0)$, and define the polynomial
\begin{align*}
    \phi_m(\lambda) := \frac{\theta_1^{(m)}}{\lambda_1} \frac{\lambda_1 - \lambda}{\theta_1^{(m)}-\lambda} p_m(\lambda).
\end{align*}
We note that $\phi_m(\lambda_1) = 0$. Since the MINRES polynomial is a minimizing polynomial, we obtain
\begin{align*}
    \|r_{m+j}\|^2 
    & \leq \|\phi_m(\mathcal{A}) \Bar{q}_j(\mathcal{A}) r_0\|^2  = \sum_{k=2}^n \phi_m^2(\lambda_k) \Bar{q}_j^2(\lambda_k) \alpha_k^2 \nonumber \leq F_m^2\sum_{k=2}^n p_m^2(\lambda_k) \Bar{q}_j^2(\lambda_k) \alpha_k^2 = F_m^2\|\Bar{r}_j\|^2, \label{r_{m+j}}
\end{align*}
where 
$
F_m = \max\limits_{k\geq 2}\frac{|\theta_1^{(m)}|}{|\lambda_1|}\frac{|\lambda_1 - \lambda_k|}{|\theta_1^{(m)}-\lambda_k|}.
$
Following~\eqref{bound-rm}, we know that 
$
    \|\bar{r}_j\| \leq 2\left(\frac{\sqrt{|ad|}- \sqrt{|bc|}}{\sqrt{|ad|} + \sqrt{|bc|}}\right)^{[j/2]} \|\bar{r}_0\|.
    $
    Then, combining with \eqref{r_{m+j}}, we obtain the bound 
\begin{align}\label{eq:residue}
    \| r_{m+j}\| \leq 2F_m\left(\frac{\sqrt{|ad|}- \sqrt{|bc|}}{\sqrt{|ad|} + \sqrt{|bc|}}\right)^{[j/2]} \|\bar{r}_0\|.
\end{align}
Therefore, the speed of convergence depends on the term $F_m$, which relates the harmonic Ritz value and the degenerate eigenvalue, as well as the distribution of the remaining bounded part of the spectrum. We demonstrate this theoretical observation for our problem in the next example.

\begin{example}[MINRES convergence revisited]\label{ex5.1}
  Using domain setup from \Cref{ex:mms}, we investigate convergence of the MINRES method for the NE and EN cases when varying the parameters $\mu{K}$. We fix $\alpha_{\text{BJS}} = 0.5$ and the mesh size $h = h_0/4$. Recall that in the NE case, the inf-sup constant $\gamma \sim \min\{\frac{1}{\sqrt{\mu {K}}},1\}$. When $\mu{K}$ is greater than $1$, a small eigenvalue appears in the spectrum, and we expect slow convergence. On the other hand, $\mu K < 1$ is the troublesome regime for convergence in the EN case as $\gamma \sim \min\{\sqrt{\mu {K}},1\}$.

 \Cref{fig:HRV-ND} and \Cref{fig:HRV-DN} show the convergence history of the preconditioned MINRES residuals, the factor $F_k$, and the values $\min_{i} |\theta_i(k) - \lambda_{\min}|$ for the NE and EN cases. Here, $\theta_i(k),  i = 1,\cdots, k$ represent the harmonic Ritz values at the $k$-th iteration.  In both cases, the residual curves show long plateau regions. However, consistent with the estimate \eqref{eq:residue}, once the $F_k$ factor converges to $1$ the plateau disappears and the convergence becomes linear.
\end{example}

\begin{figure}[htbp]
  \setlength{\abovecaptionskip}{0.0625\baselineskip}    
  \setlength{\belowcaptionskip}{0.0625\baselineskip}    
\centering
\begin{subfigure}{0.45\textwidth}
\includegraphics[width=0.8\textwidth]{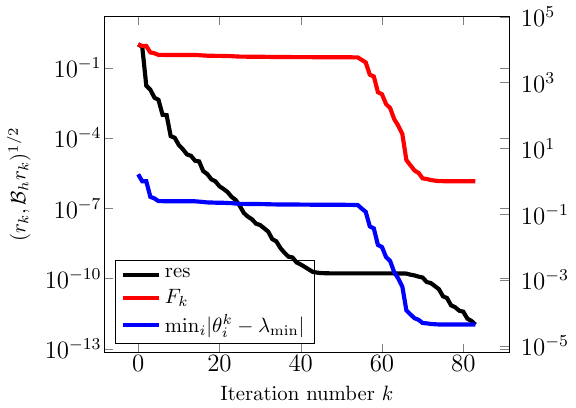}
\caption{$\mu K = 10^4$}
\label{fig:HRV-ND-a}
\end{subfigure}
\begin{subfigure}{0.45\textwidth}
\includegraphics[width=0.8\textwidth]{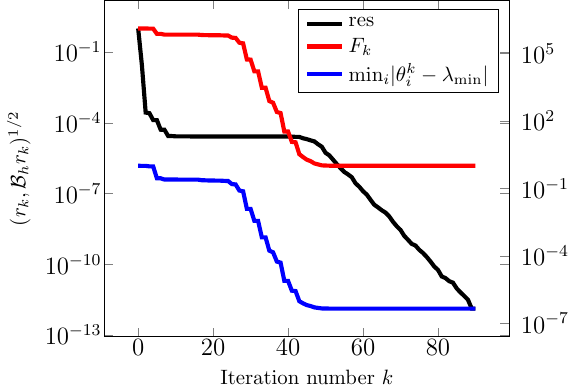} 
\caption{$\mu K = 10^{6}$}
\label{fig:HRV-ND-b}
\end{subfigure}
\caption{Convergence history of the MINRES method (plotted against the left vertical-axis), the harmonic Ritz value, and
  the term $F_k$ (both plotted against the right vertical axis) for the NE case when $\alpha_{\text{BJS}} = 0.5$ and $h = h_0/4$.}
\label{fig:HRV-ND}
\end{figure}

\begin{figure}[htbp]
  \setlength{\abovecaptionskip}{0.0625\baselineskip}    
  \setlength{\belowcaptionskip}{0.0625\baselineskip}  
\centering
\begin{subfigure}{0.45\textwidth}
\includegraphics[width=0.8\textwidth]{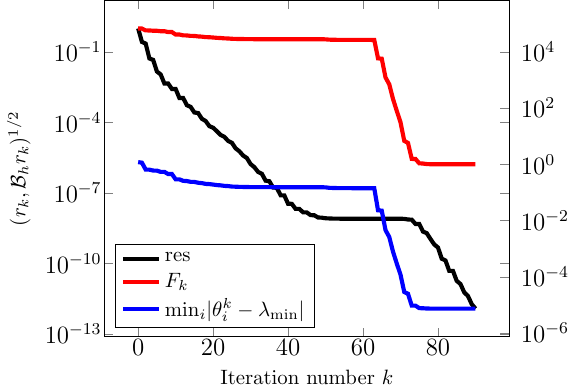}
\caption{$\mu K = 10^{-6}$}
\label{fig:HRV-DN-a}
\end{subfigure}
\begin{subfigure}{0.45\textwidth}
\includegraphics[width=0.8\textwidth]{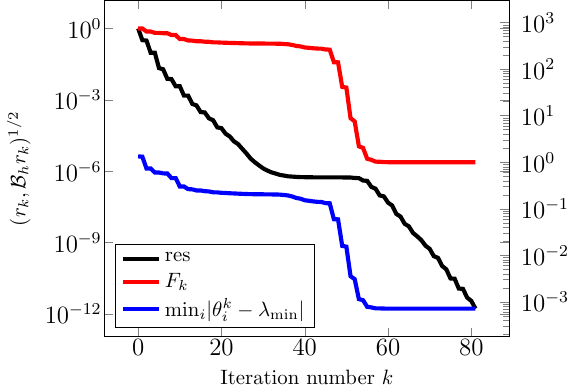} 
\caption{$\mu K = 10^{-4}$}
\label{fig:HRV-DN-b}
\end{subfigure}
\caption{Convergence history of the MINRES method (plotted against the left vertical axis), the harmonic Ritz value,
  and the term $F_k$ (plotted against the right vertical axis) for the EN case with $\alpha_{\text{BJS}} = 0.5$ and $h = h_0/4$.}
\label{fig:HRV-DN}
\end{figure}

In the next section, we employ the deflation technique (e.g. \cite{coulaud2013deflation}) to remove the plateau region and accelerate the MINRES convergence.

\subsection{Deflation technique}\label{sec:deflation}
Let $x_k$ be the $k$-th iterate of the MINRES method for \eqref{iter-system}, with $r_k$ as the corresponding residual. Given a projection operator $P_{\mathcal{A}}$ onto the slow-to-converge eigenmodes, a deflation preconditioner is applied as an update
\[
x_{k+1} = x_k + P_{\mathcal{A}} (P_{\mathcal{A}}^{T} \mathcal{A} P_{\mathcal{A}})^{-1} P_{\mathcal{A}}^{T} r_k.
\]
Constructing $P_{\mathcal{A}}$ in the case of the Stokes-Darcy problem requires computing the eigenmode associated with the smallest eigenvalue in magnitude of $\mathcal{A}x=\lambda \mathcal{B}^{-1} x$, where $\mathcal{B}$ is
the preconditioner \eqref{preconditioner}.  However, this approach is computationally expensive and impractical. Thus, to avoid this, we propose an alternative update
\begin{equation}\label{eq:update}
x_{k+1} = x_k + P_{W} (P_{W}^{T} \gamma \mathcal{B}^{-1} P_{W})^{-1} P_{W}^{T} r_k,
\end{equation}
where $P_W$ is a projection operator onto the subspace spanned by approximations of the slow-to-converge eigenmodes and $\gamma>0$ is a problem-dependent scaling parameter.

\begin{figure}[h!]
  \centering
  \includegraphics[width=0.7\textwidth]{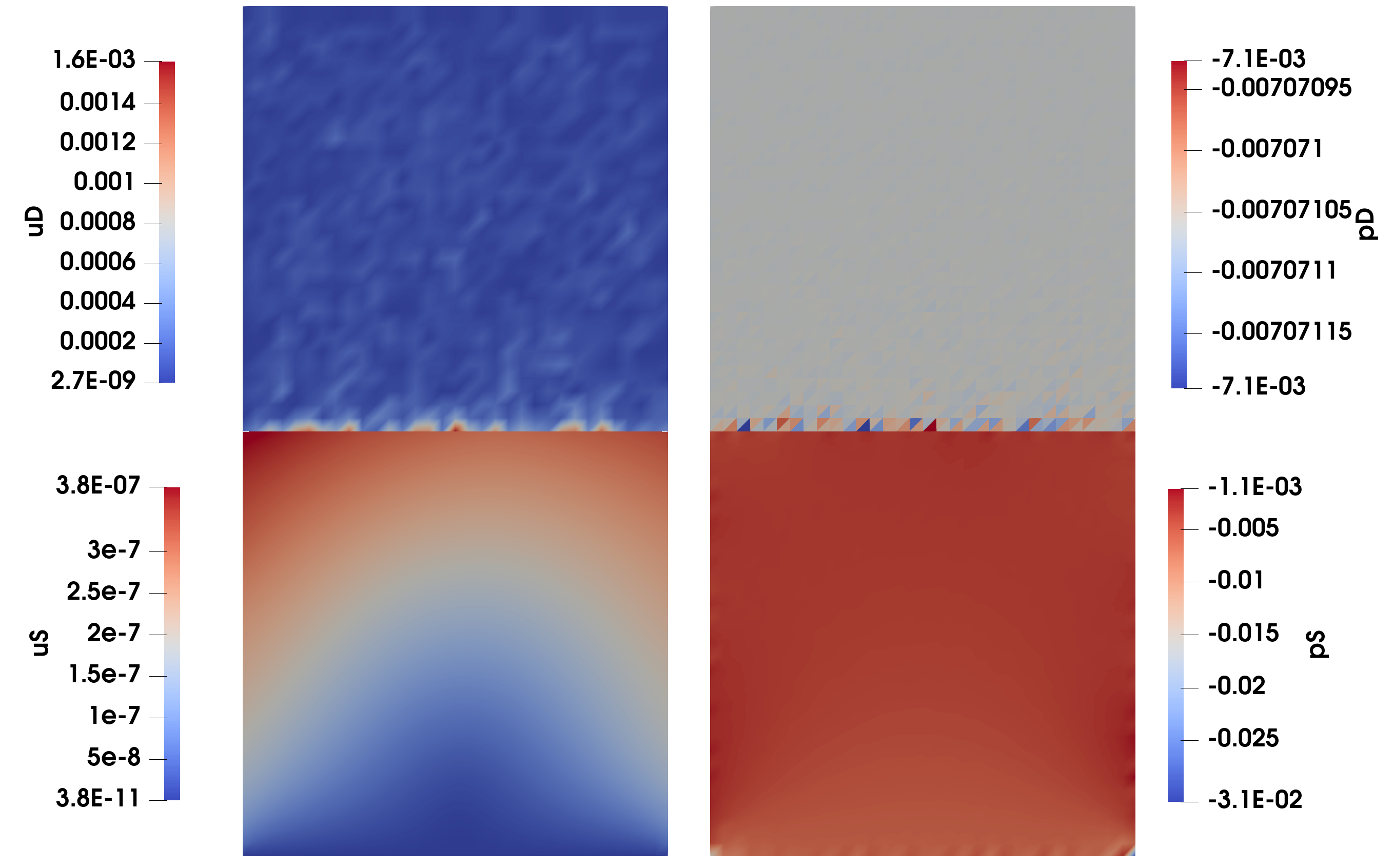}
  \caption{Discrete eigenmode corresponding to the inf-sup constant for the Stokes-Darcy problem, applied to the geometry from \Cref{ex:mms}, for the
    NE case and $\mu=K=10^4$. Components $\mathbf{u}_{S, h}$ and $\mathbf{u}_{D, h}$ are small almost everywhere. Away from the interface $p_{D, h}\approx \text{const} > 0$ 
    while $p_{S, h}\approx 0$. Without including its plot, we remark that $\lambda_h$ is approximately constant.
    Therefore, we use approximation $w_h\approx(\mathbf{0}, \mathbf{0}, 0, 1, 1)$ in the deflation.}
  \label{fig:inf_sup}
\end{figure}

To motivate the approach \eqref{eq:update}, \Cref{fig:inf_sup} illustrates the eigenmode for which the inf-sup constant is attained in the Stokes-Darcy problem from \Cref{ex:mms} for the NE case. As suggested by \cref{thm-depend-para}, the function $w=(\mathbf{u}_S, \mathbf{u}_D, p_S, p_D, \lambda)=(\mathbf{0}, \mathbf{0}, 0, 1, 1)$ provides a good approximation of the true eigenmode.   Therefore, we define $P_W$ and the scaling parameter as follows
\[
w = \begin{cases}
     (\mathbf{0}, \mathbf{0}, 0, 1, 1)    , \ \text{NE} \ \text{case}; \\
     (\mathbf{0}, \mathbf{0}, 1, 0, 1), \  \text{EN}\ \text{case}.
\end{cases}, \quad
    \gamma \sim 
    \begin{cases}
       (\mu {K})^{-1}, \ \text{NE} \ \text{case}; \\
       (\mu {K}), \  \text{EN}\ \text{case},
    \end{cases}
    \]
    and, combing with the preconditioner \eqref{preconditioner}, the deflation preconditioner is defined as
    \begin{equation}\label{eq:defl_preconditioner}
\mathcal{B}_W = \mathcal{B} + P_{W} (P_{W}^{T} \gamma \mathcal{B}^{-1} P_{W})^{-1} P_{W}^{T}.
    \end{equation}
    We remark that the scaling parameter $\gamma$ reflects the dependence of the inf-sup constant in \Cref{thm-depend-para}.  Since the near-kernel is one dimensional, the matrix representation of $P_W$ is a matrix consisting of a single column, where the column vector represents  the function $w$ in $W_h$. As as result, computing the small $(1\times 1)$ matrix $(P_{W}^{T} \gamma \mathcal{B}^{-1} P_{W})$ and its inverse can be done efficiently with optimal complexity. 

    \begin{example}[Deflation preconditioner]\label{ex:deflation}
      To demonstrate the improved performance of the deflation preconditioner \eqref{eq:defl_preconditioner}, we repeat the experiments from \Cref{ex:ker}. As shown in \Cref{fig:history_ker_defl}, the new preconditioner eliminates the plateau regions in the MINRES convergence curves, leading to significantly faster convergence in parameter regimes that previously exhibited slow convergence with \eqref{preconditioner}. In other regimes, \eqref{eq:defl_preconditioner} maintains the favorable performance of \eqref{preconditioner}. Finally, \Cref{fig:cond_ker_defl} illustrates the dependence of the condition number on parameters and mesh size when using \eqref{eq:defl_preconditioner}. Comparing this with \Cref{fig:cond_ker}, which corresponds to \eqref{preconditioner}, we observe that deflation effectively improves the system’s conditioning.
    \end{example}

\begin{figure}[]
  \setlength{\abovecaptionskip}{0.0625\baselineskip}    
  \setlength{\belowcaptionskip}{0.0625\baselineskip}
  \centering
    \begin{subfigure}[c]{\textwidth}
      \includegraphics[width=\textwidth]{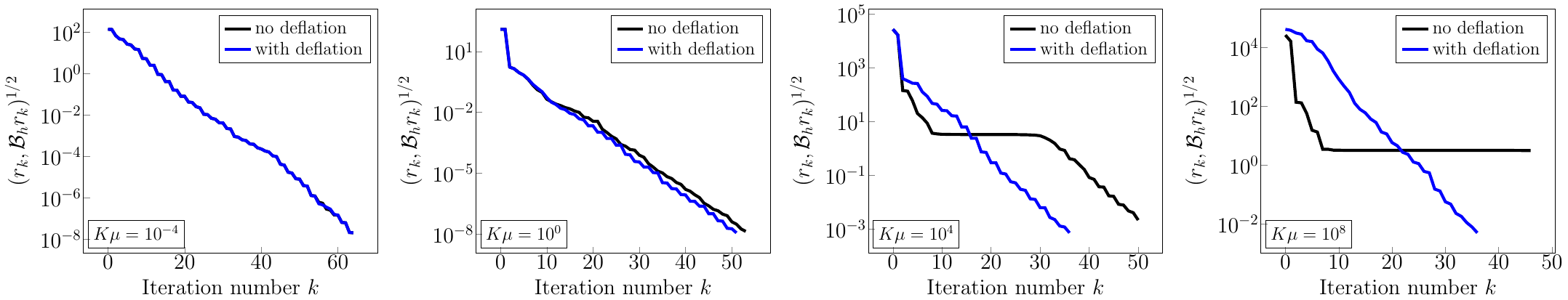}
    \caption{NE}\label{fig:NE_history_ker_defl}
    \end{subfigure}
    \begin{subfigure}[c]{\textwidth}
      \includegraphics[width=\textwidth]{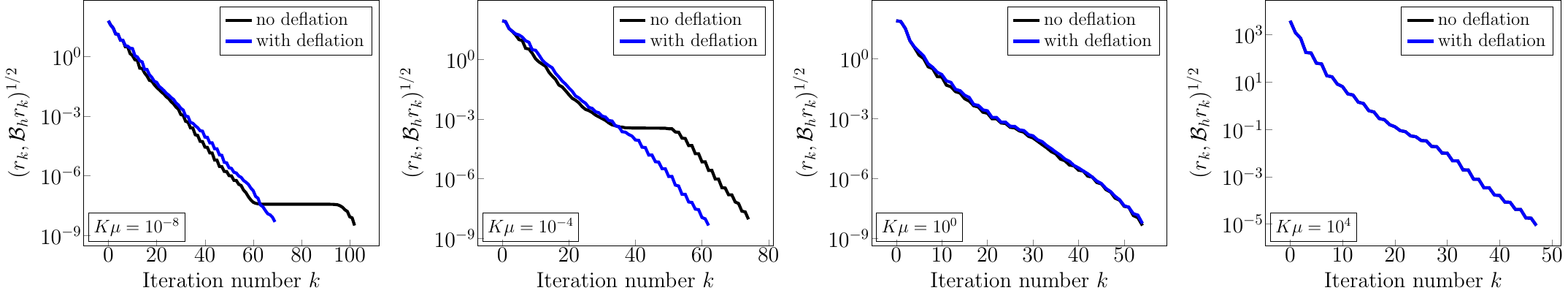}
    \caption{EN}\label{fig:EN_history_ker_defl}
    \end{subfigure}
    \caption{MINRES convergence history using deflation preconditioner \eqref{eq:defl_preconditioner} and
      the standard preconditioner \eqref{preconditioner}. 
      Setup from \Cref{ex:mms} is considered with a fixed mesh resolution $h=h_0/8$ (cf. \Cref{fig:cvrg_EN}).
  }
  \label{fig:history_ker_defl}
\end{figure}

\begin{figure}[]
  \setlength{\abovecaptionskip}{0.0625\baselineskip}    
  \setlength{\belowcaptionskip}{0.0625\baselineskip}
  \centering
    \begin{subfigure}[c]{\textwidth}
      \includegraphics[width=\textwidth]{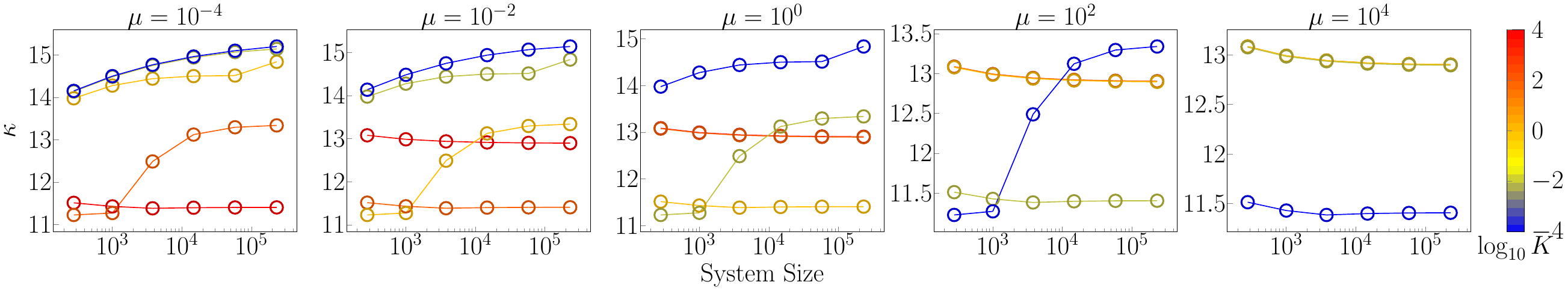}
    \caption{NE}\label{fig:NE_cond_ker_defl}
    \end{subfigure}
    \begin{subfigure}[c]{\textwidth}
      \includegraphics[width=\textwidth]{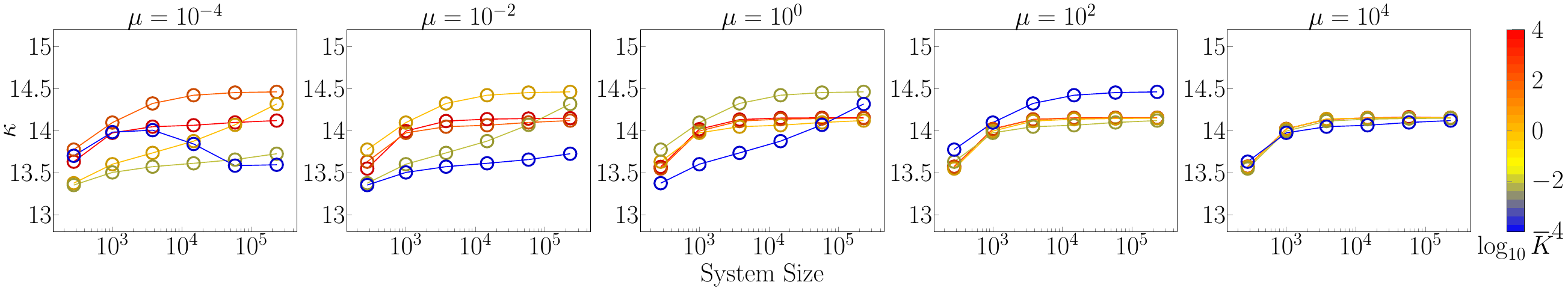}
    \caption{EN}\label{fig:EN_cond_ker_defl}
    \end{subfigure}
    \caption{Performance of preconditioner \eqref{eq:defl_preconditioner} for the
      Stokes-Darcy problem, applied on the geometry from \Cref{ex:mms}, with different parameter
      values for the EN and NE cases (see also \Cref{fig:bc_config}). Deflation
      effectively removes the parameter sensitivity observed in \Cref{fig:cond_ker}.
    }
  \label{fig:cond_ker_defl}
\end{figure}

In the final example we aim to demonstrate that the preconditioner \eqref{preconditioner} and 
the deflation technique \eqref{eq:defl_preconditioner} can be adapted in the 
settings beyond those analyzed theoretically in \Cref{sec3:well-posed}. In particular, 
we shall consider the model setup with floating Darcy domains, that is, where the boundary 
of the Darcy domains coincides with the interface, i.e. $\partial\Omega_D=\Gamma$. 

\begin{example}[Floating Darcy subdomains]\label{ex:floating}
  We study the Stokes-Darcy problem in geometries with multiple disconnected porous subdomains, This setup is inspired by \cite{bukavc2023analysis} in the context of tissue bioengineering,  particularly in the design of scaffolds modeled as porous structures. As a representative example, we consider a channel domain containing $12$ Darcy subdomains, denoted as $\Omega^i_D$ for $i=1, \dots, 12$. Each subdomain has a hexagonal shape with a radius of $1$ and is arranged such that the free-flow region forms channels of width $0.2$, as illustrated in \Cref{fig:floating_setup}. We impose a no-slip condition, $\mathbf{u}_S=0$, on the top and bottom walls of the channel, while the flow is driven by a traction boundary condition of the form $-\mathbf{\sigma}(\mathbf{u}_S, p_S)\cdot\mathbf{n}_S=p\mathbf{n}S$, where $p=1$ on the left edge and $p=0$ on the right edge. Although the coupled problem remains well-posed, the absence of a Neumann boundary on the Darcy subdomains ($\lvert\Gamma^N_D\rvert=0$) poses a preconditioning challenge similar to the NE case in \Cref{thm-depend-para}. To illustrate this issue, we fix $ \mu=3 $ and $\alpha_{\text{BJS}}=0.5$ while varying $K\in\left\{1, 100\right\}$. The resulting flow field for $K=100$ is depicted in \Cref{fig:floating_setup}.
  
\begin{figure}
  \centering
      \includegraphics[height=0.32\textwidth, align=c]{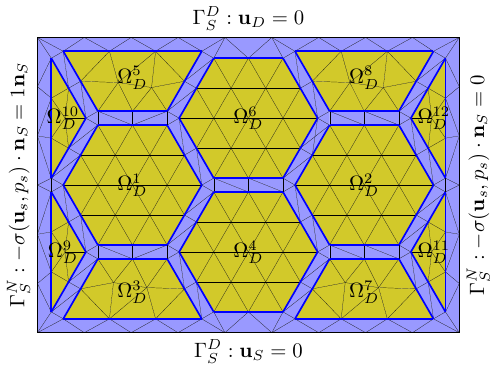}
  \includegraphics[height=0.285\textwidth, align=c]{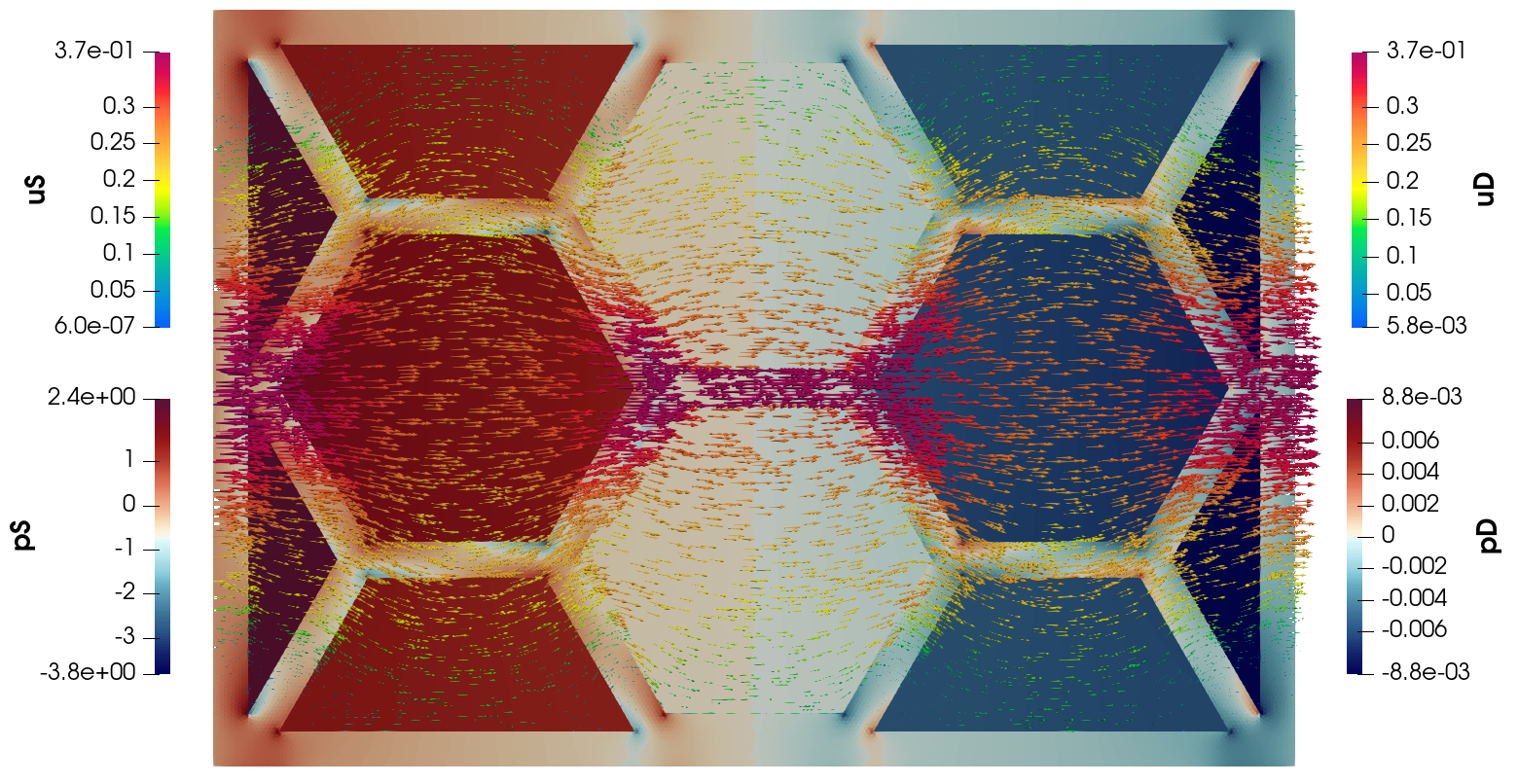}
  \caption{
    Stokes-Darcy model with floating Darcy subdomains. (Left) Problem setup together with initial mesh.
    (Right) Pressure and velocity fields for $\mu=3$, $\alpha_{\text{BJS}}=0.5$, $K=100$.
  }
  \label{fig:floating_setup}
\end{figure}

Applying the preconditioner \eqref{preconditioner}, \Cref{fig:floating_cvrg} shows the different MINRES convergence behavior for $K=1$ and $K=100$. Notably, the iterations are sensitive to $K$, with the $K=100$ case exhibiting multiple plateaus.  As in \Cref{thm-depend-para}, the parameter-robust inf-sup constant would require us consider the problem on the subspace $\left\{p\in L^2(\Omega_D): (p, \chi_{\Omega^i_D})_{\Omega_D}=0,\, i=1, \dots, 12\right\}$, where $\chi_{\Omega^i_D}$ is the indicator function of $\Omega^i_D$.  To address this, we apply the deflated preconditioner \eqref{eq:defl_preconditioner} with the deflation subspace spanned by vectors $w_i=(\mathbf{u}_S, \mathbf{u}_D, p_S, p^i_D, \lambda^i)$, where $\mathbf{u}_S=0$, $\mathbf{u}_D=0$, $p_S=0$, $p^i_D=\chi_{\Omega^i_D}$, $\lambda^i=p^i_D|_{\Gamma}$. As in the NE case, we set $\gamma \sim (\mu K)^{-1}$. \Cref{fig:floating_cvrg} confirms that deflation eliminates slow convergence regions, leading to iteration counts that remain stable. 

\begin{figure}
  \centering
  \includegraphics[width=\textwidth]{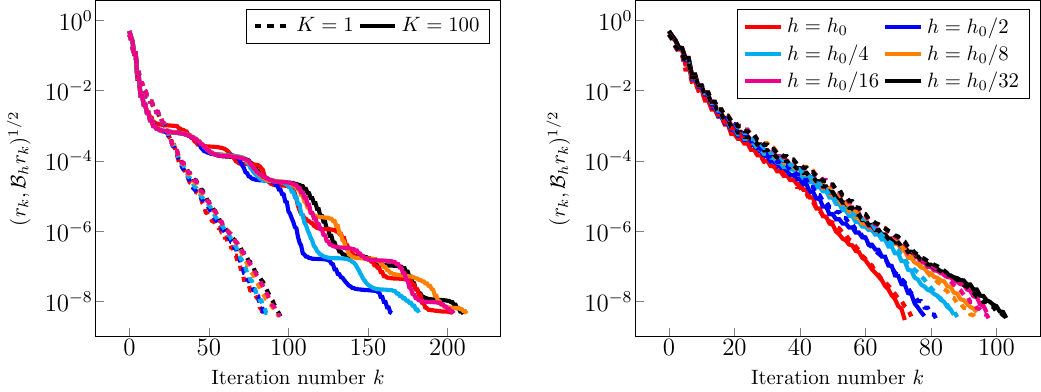}
  \caption{
    Convergence history for preconditioned MINRES solver on the Stokes-Darcy problem with $1\leq i \leq 12$ floating domains, cf. \Cref{fig:floating_setup}.
    (Left) Preconditioner \eqref{preconditioner} without deflation. (Right) Deflation preconditioner \eqref{eq:defl_preconditioner}
    is used with deflated subspace spanned by vectors $\mathbf{u}_S=0$, $\mathbf{u}_D=0$, $p_S=0$, $p^i_D=\chi_{\Omega^i_D}$, $\lambda^i=p^i_D|_{\Gamma}$.
    In both plots, colors encode refinement level with the mesh of $h=h_0$ shown
    in \Cref{fig:floating_setup} and the results for $K=1$ are shown in dashed line.
  }
  \label{fig:floating_cvrg}
\end{figure}
\end{example}

\section{Conclusions}\label{sec:conclusions}
We studied a coupled Stokes-Darcy problem formulation, using a Lagrange multiplier to enforce mass conservation at the interface, and established the well-posedness of the problem for various boundary condition configurations. We demonstrated that in certain configurations, uniform stability estimates with respect to the model parameters require a constrained solution space to avoid the near-kernel issue. Through the analysis of the MINRES method’s convergence, we showed that the presence of the near kernel can lead to stagnation. Based on these theoretical findings, we developed parameter-robust preconditioners, which, particularly in near-kernel cases, employ deflation techniques to accelerate convergence.

\bibliographystyle{siam}
\bibliography{reference}

\end{document}